\documentclass[10pt,a4paper]{amsart}
\usepackage{amsthm,amsfonts,amsmath,amssymb,latexsym,amscd}
\usepackage{epsfig,graphics,color}
\usepackage{psfrag} %added by EV.
\usepackage{pinlabel} %added by EV.
\usepackage{graphicx}  %added by EV.
\usepackage[matrix,arrow,curve]{xy}
\usepackage{hyperref}
\usepackage{mathrsfs}
\usepackage{verbatim}
\usepackage{bm}
\usepackage{enumerate}
\title{Eight flavours of cyclic homology}
\author{K.~Cieliebak and E.~Volkov}
\theoremstyle{plain}
\newtheorem{theorem}{Theorem}[section]
\newtheorem{thm}[theorem]{Theorem}
\newtheorem{corollary}[theorem]{Corollary}
\newtheorem{cor}[theorem]{Corollary}
\newtheorem{proposition}[theorem]{Proposition}
\newtheorem{prop}[theorem]{Proposition}
\newtheorem{lemma}[theorem]{Lemma}
\newtheorem{lem}[theorem]{Lemma}

\theoremstyle{remark}

\newtheorem{remark}[theorem]{Remark}

\newtheorem{example}[theorem]{Example}

\newtheorem{definition}{Definition}

\newcommand{\id}{{\rm id}}

\newcommand{\ol}{\overline}

\newcommand{\p}{\partial}

\newcommand{\om}{\omega}
\newcommand{\Om}{\Omega}
\newcommand{\eps}{\varepsilon}
\newcommand{\into}{\hookrightarrow}
\newcommand{\la}{\langle}
\newcommand{\ra}{\rangle}
\newcommand{\wt}{\widetilde}
\newcommand{\wh}{\widehat}
\newcommand{\N}{{\mathbb{N}}}
\newcommand{\Z}{{\mathbb{Z}}}
\newcommand{\R}{{\mathbb{R}}}

\newcommand{\m}{{\bf m}}

\newcommand{\im}{{\rm im\,}}        % image

\newcommand{\vol}{{\rm vol}}

\newcommand{\pt}{{\rm pt}}

\newcommand{\sign}{{\rm sign}}

\newcommand{\Diag}{{\rm Diag}}

\newcommand{\ver}{{\rm ver}}
\newcommand{\hor}{{\rm hor}}
\newcommand{\ev}{{\rm ev}}
\newcommand{\cyc}{{\rm cyc}}
\newcommand{\cp}{{\rm cp}}
\newcommand{\red}{{\rm red}}
\newcommand{\Hom}{{\rm Hom}}

\newcommand{\FF}{\mathcal{F}}

\newcommand{\HH}{\mathcal{H}}

\newcommand{\mm}{{\mathfrak m}}

\newcommand{\ff}{{\mathfrak f}}
\newcommand{\fm}{{\mathfrak m}}
\newcommand{\fn}{{\mathfrak n}}

\begin{document}

\begin{abstract}
We introduce eight ``flavours'' of cyclic homology of a
mixed complex and study their properties. In particular, we determine
their behaviour with respect to Chen's iterated integrals. 
\end{abstract}

\maketitle
%\tableofcontents

%%%%%%%%%%%%%%%%%%%%%%%%%%%%%%%%%%%%%%%%%%%%%%%%%%%%%%%%%%%%%%%%%%%%%%%%%%%%
\section{Introduction}
%%%%%%%%%%%%%%%%%%%%%%%%%%%%%%%%%%%%%%%%%%%%%%%%%%%%%%%%%%%%%%%%%%%%%%%%%%%%

Cyclic homology of an algebra was introduced in the mid-1980s by
B.~Tsygan~\cite{Tsygan} and A.~Connes~\cite{Connes}. It can be seen as
an algebraic counterpart to the $S^1$-equivariant homology of a space
with a circle action. Since then, cyclic homology has been generalized
to differential graded algebras (dgas), $A_\infty$-algebras, and beyond.
Moreover, several other versions of cyclic homology have emerged:
negative cyclic homology (Jones~\cite{Jones}), periodic cyclic
homology (Goodwillie~\cite{Goodwillie}), and completed negative cyclic
homology (Jones and Petrack~\cite{Jones-Petrack}). 

Cyclic homology is related to the homology of loop spaces in two
different ways.
First, for each connected space $X$, a suitable version of cyclic homology
of singular chains on the based loop space of $X$ (made a dga by the
Pontrjagin product) is isomorphic to the $S^1$-equivariant homology of
its free loop space $LX$.
The second relation, which was the starting point of this paper, 
goes back to the work by K.T.~Chen~\cite{Chen73,Chen77}. For a simply connected
manifold $X$, Chen showed that the singular cohomology of its based
loop space can be computed in terms of iterated integrals of differential
forms. Getzler, Jones and Petrack~\cite{Getzler-Jones-Petrack}
extended this result to the homology and the $S^1$-equivariant
homology of the free loop space $LX$. 

The goal of the present paper is a systematic study of the different
versions of cyclic homology and their relation to loop space homology
via Chen's iterated integrals. The natural setting is that of a {\em
  mixed complex}, introduced by Kassel~\cite{Kassel87} (and popularized
by Getzler, Jones and Petrack~\cite{Getzler-Jones-Petrack} under the
name dg-$\Lambda$-module). This is a $\Z$-graded vector space $C$
%$C=\bigoplus_{k\in\Z}C^k$
together with two anticommuting
differentials $\delta,D$ of degrees $|\delta|=1$ and $|D|=-1$.
The two main examples are the Hochschild complex $(C(A),d_H,B)$ 
of a dga $A$ with its Connes operator $B$, and the singular cochain
complex $(C^*(Y),d,P)$ of an $S^1$-space with its contraction $P$ by
the circle action. See Section~\ref{sec:mixed} for details.

For a mixed complex $(C,\delta,D)$, the map $\delta_u=\delta+uD$
defines a differential of degree $1$ on the space $C[[u,u^{-1}]]$
of formal power series in a degree $2$ variable $u$ and its inverse.
This complex has five subcomplexes $C[[u,u^{-1}]$, $C[u,u^{-1}]]$,
$C[u,u^{-1}]$, $C[[u]]$ and $C[u]$, corresponding to power series that
are polynomial in $u^{-1}$ etc, and two quotient complexes
$C[[u^{-1}]]=C[[u,u^{-1}]]/uC[[u]]$ and $C[u^{-1}]=C[u,u^{-1}]/uC[u]$.
{\let\thefootnote\relax\footnote{{\it Mathematics Subject Classification:} Primary 55N91}}
{\let\thefootnote\relax\footnote{{\it Keywords:} Cyclic homology, equivariant cohomology}}
These give rise to eight versions (or ``flavours'') of cyclic homology
that we denote by $HC^{[[u,u^{-1}]]}$ etc.\footnote{
Note that the differential $\delta_u$ computing cyclic homology has
degree $+1$; see Remark~\ref{rem:cyc-hom}(c) for further discussion of
our grading covention.
}
The eight versions of cyclic homology are in general all different,
and they are all invariant under homotopy equivalences of mixed complexes.
However, only the three versions $HC^{[u^{-1}]}$, $HC^{[[u]]}$ and
$HC^{[[u,u^{-1}]}$ are invariant under
quasi-isomorphisms of mixed complexes (Proposition~\ref{prop:quism-invariance}).
These correspond to the positive, negative and periodic cyclic
homologies in~\cite{Jones}, and we will refer to them as the {\em
  classical versions}. 

For the Hochschild complex of a dga (or more generally of a cyclic
cochain complex), each version of cyclic homology is either trivial or
agrees with one of the $3$ classical versions (Corollary~\ref{cor:cyccochain}).
Moreover, the version $HC^{[u^{-1}]}$ agrees with {\em Connes' version}
$HC^\lambda$, defined as the $d_H$-homology of the Hochschild complex
modulo cyclic permutations.

For the singular cochain complex of a smooth $S^1$-space $Y$ (such as the
free loop space of a manifold), the version $HC^{[[u]]}$ of cyclic homology
agrees with the singular cohomology $H^*_{S^1}(Y)$ of its Borel space
$(Y\times ES^1)/S^1$ (Jones~\cite{Jones}). 
The (non-classical) version $HC^{[[u,u^{-1}]]}$ satisfies fixed point localization
(Jones and Petrack~\cite{Jones-Petrack}), whereas the version
$HC^{[[u,u^{-1}]}$ for a free loop space $Y=LX$ depends only on the
fundamental group of $X$ (Goodwillie~\cite{Goodwillie}). 

Consider now a manifold $X$ with its de Rham dga $\Om^*(X)$. According
to Getzler, Jones and Petrack~\cite{Getzler-Jones-Petrack} (see also
Proposition~\ref{prop:chen} below), Chen's iterated integrals define a
morphism of mixed complexes 
$$
   I: \bigl(C(\Om^*(X)),d_H,B\bigr)\to \bigl(C^*(LX),d,P\bigr).
$$
For $X$ simply connected this map is a quasi-isomorphism, so it
induces isomorphisms on the three classical versions of cyclic homology.
On the other five versions it does not induce an isomorphism in general. 
The main result of this paper (Corollary~\ref{cor:main}) asserts that
for a simply connected manifold $X$ the cyclic (or Connes) variant
$I^\lambda$ of Chen's iterated integral induces an isomorphism
\begin{align*}
   I^\lambda_*: \ol{HC}^\lambda_*(\Om^*(X)) \stackrel{\cong}\longrightarrow H^*_{S^1}(LX,x_0)
\end{align*}
between the reduced Connes version of cyclic homology of the de Rham
complex and the $S^1$-equivariant cohomology of $LX$ relative
to a fixed constant loop $x_0$. 

All the preceding results have counterparts for cyclic {\em cohomology}.
For example (Corollary~\ref{cor:main-dual}),
for a simply connected manifold $X$ the map $J_\lambda$ adjoint to
$I^\lambda$ induces an isomorphism
$J_{\lambda*}: H_*^{S^1}(LX,x_0) \stackrel{\cong}\longrightarrow
\ol{HC}_\lambda^*(\Om^*(X))$. 
\medskip

The motivation for this article comes from {\em string topology}. This
term refers to algebraic structures on loop space homologies 
introduced by Chas and Sullivan in~\cite{Chas-Sullivan99} and subsequent work.
One of the puzzles in string topology concerns the appropriate
versions of loop space homology on which these structures are defined.
In the non-equivariant case, it has recently turned out that the
Chas--Sullivan loop product and the Goresky--Hingston coproduct
both naturally live on $H_*(LX)$ if the manifold $X$ has vanishing
Euler characteristic, and on $H_*(LX,x_0)$ otherwise~\cite{Cieliebak-Hingston-Oancea-PD}.  

In the equivariant case, Chas and Sullivan described in~\cite{Chas-Sullivan04}
an involutive Lie bialgebra structure on $H^{S^1}_*(LX,X)$, the
$S^1$-equivariant homology of the loop space $LX$ relative to the
subset $X\subset LX$ of constant loops. 
In~\cite{Cieliebak-Fukaya-Latschev} it was conjectured that a chain-level version of this
structure exists on $HC_\lambda^*(\Om^*(X))$, the Connes version of
cyclic cohomology of the de Rham complex, so the question arose what 
this corresponds to on the loop space side. This question became more
pressing when computations of examples showed that $HC_\lambda^*(\Om^*(X))$
can be nontrivial in negative degrees, and thus cannot correspond to
any version of loop space homology. The solution is provided by
Corollary~\ref{cor:main-dual}: the negative degree part in $HC_\lambda^*(\Om^*(X))$ 
only comes from the cohomology of a point, and after dividing this out (i.e.,
passing to reduced cohomology) it becomes isomorphic to $H_*^{S^1}(LX,x_0)$.
In particular, this exhibits loop space homology relative to a point
as the natural space supporting the involutive Lie bialgebra structure
from~\cite{Chas-Sullivan04}. 
\medskip

{\bf Acknowledgements. }
We thank P.~H\'ajek for many stimulating discussions, and N.~Bondarenko for drawing the figures. 

%%%%%%%%%%%%%%%%%%%%%%%%%%%%%%%%%%%%%%%%%%%%%%%%%%%%%%%%%%%%%%%%%%%%%%%%%%%%
\section{Cyclic homology of mixed complexes}\label{sec:mixed} 
%%%%%%%%%%%%%%%%%%%%%%%%%%%%%%%%%%%%%%%%%%%%%%%%%%%%%%%%%%%%%%%%%%%%%%%%%%%%

In this section we introduce the $8$ versions of cyclic homology
associated to a mixed complex and discuss their properties. Moreover,
we will establish the following diagram in which all maps are
functors and the upper %right square and triangle
triangles commute, whereas the
lower pentagon commutes for simply connected manifolds. %\footnote{
(The map from dgas to Connes double complexes actually factors through
$A_\infty$-algebras, but we will not discuss those in this paper.)

\begin{equation}\label{eq:functors}
\xymatrix{
&\text{dga} \ar[d] \ar[dr] \\ % &\text{$A_\infty$-algebra} \ar[d] \\ 
\text{manifold} \ar[ur]^{\text{de Rham}} \ar[r] \ar[dr]_{\text{loop space}} &\text{cyclic cochain complex} \ar[r] &\text{Connes double complex} \ar[d] \\ 
&\text{$S^1$-space} \ar[r]^{\text{singular cochains}} &\text{mixed complex} \ar[d]_{\text{cyclic homologies}} \\ 
&&\text{$\R[u]$-module} \\
%&\text{dga} \ar[d] \ar[r] &\text{$A_\infty$-algebra} \ar[d] \\ 
%\text{manifold} \ar[ur]^{\text{de Rham}} \ar[r] \ar[dr]_{\text{loop space}} &\text{cyclic cochain complex} \ar[r] &\text{Connes double complex} \ar[d] \\ 
%&\text{$S^1$-space} \ar[r]^{\text{singular cochains}} &\text{mixed complex} \ar[d]_{\text{cyclic homologies}} \\ 
%&&\text{$\R[u]$-module} \\
}
\end{equation}

%%%
\subsection{Mixed complexes}
%%%

\begin{definition}
A {\em mixed complex} $(C,\delta,D)$ is a $\Z$-graded $\R$-vector
space\footnote{Most of this section works for modules over a
  commutative ring with unit instead of $\R$-vector spaces.}
$$
  C=\bigoplus_{k\in\Z}C_k
$$ 
with two linear maps $\delta,D:C\to C$ of degrees $|\delta|=1$ and $|D|=-1$ satisfying 
$$
   \delta^2=0,\qquad D^2=0,\qquad \delta D+D\delta=0. 
$$
A {\em morphism} between mixed complexes is a linear map $f:C\to\wt
C$ of degree $0$ satisfying 
%\marginpar{\small The second relation can be generalized, see~\cite{Zhao}. We
%need to see which notion is appropriate for us.}
$$
   \wt \delta f=f\delta,\qquad \wt Df=fD. 
$$
It is called a {\em quasi-isomorphism} if it induces an isomorphism
on homology $H(C,\delta)\to H(\wt C,\wt \delta)$. 
A {\em homotopy} between two morphisms $f,g:C\to\wt C$ is a linear
map $H:C\to\wt C$ of degree $-1$ such that $\delta H+H\delta=f-g$ and $DH+HD=0$.
A morphism $f:C\to\wt C$ is called a {\em homotopy equivalence} if
there exists a morphism $g:\wt C\to C$ such that $fg$ and $gf$ are
homotopic to the identity. 
Every homotopy equivalence is a quasi-isomorphism but not vice versa. 
\end{definition}

Let $u$ be a formal variable of degree $|u|=2$. To a mixed complex $(C,\delta,D)$
we associate the cochain complex 
$$
   C[[u,u^{-1}]]:=\bigoplus_{k\in\Z}C_k[[u,u^{-1}]],\qquad
   \delta_u:=\delta+uD.
$$
where $C_k[[u,u^{-1}]]$ denotes the space of formal power series
$\sum_{i\in\Z}c_iu^i$ with $c_i\in C_{k-2i}$. We emphasize that
$C[[u,u^{-1}]]$ is {\em not} the usual tensor product of $C$ with
$\R[[u,u^{-1}]]$. Note that $\delta_u$ has degree $+1$. This complex has
seven sub/quotient complexes of interest, all equipped with the
differential induced by $\delta_u$:
\begin{equation}\label{eq:cyc-versions}
\begin{gathered}
   C[[u,u^{-1}],\quad C[u,u^{-1}]],\quad C[u,u^{-1}],\quad
   C[[u]],\quad C[u], \cr 
   C[[u^{-1}]] := C[[u,u^{-1}]]/uC[[u]] = C[u,u^{-1}]]/uC[u], \cr
   C[u^{-1}] := C[[u,u^{-1}]/uC[[u]] = C[u,u^{-1}]/uC[u].
\end{gathered}
\end{equation}
Here the complexes in the first line are the obvious subcomplexes of
$C[[u,u^{-1}]]$, where $C[[u,u^{-1}]$ denotes power series in $u$
and polynomials in $u^{-1}$ etc., and the remaining two complexes are
quotients. Note that 
$$
   C[u,u^{-1}] = u^{-1}C[u],\qquad C[[u,u^{-1}] = u^{-1}C[[u]],
$$
where the right hand sides denote the localization of $C[u]$
(resp.~$C[[u]]$) at the multiplicative set $\{1,u,u^2,\dots\}$. 
We denote the homology of $C[[u,u^{-1}]]$ with respect to $\delta_u$ by 
$$
   HC^{[[u,u^{-1}]]} := H\Bigl(C[[u,u^{-1}]],\delta_u\Bigr),
$$
and similarly for the other versions with the obvious notation. 
By construction, all the chain complexes and thus also their
homologies are modules over the polynomial ring $\R[u]$. Moreover, the 
versions $C[[u,u^{-1}]]$, $C[[u,u^{-1}]$, $C[u,u^{-1}]]$,
$C[u,u^{-1}]$ and their cohomologies are modules over the larger ring
$\R[u,u^{-1}]$ of Laurent polynomials. 
%\begin{definition}
We will use the following names for some versions of cyclic homology:
\vspace{-5mm}
\begin{itemize}
\item $HC^{[[u]]}$ {\em Borel version};
\item $HC^{[[u,u^{-1}]}$ {\em Goodwillie version};
\item $HC^{[u^{-1}]}$ {\em nilpotent version};
\item $HC^{[[u,u^{-1}]]}$ {\em Jones--Petrack version}.
\end{itemize}
%\end{definition}
The first three of these versions will also be called the {\em
  classical versions}. 

\begin{remark}\label{rem:cyc-hom}
(a) The explanations for the preceding names are the
following (see later in this section for details): 
$HC^{[[u]]}$ applied to cochains on an $S^1$-space yields the
cohomology of its Borel construction;
$HC^{[[u,u^{-1}]}$ satisfies Goodwillie's theorem: applied to cochains
on a loop space $LX$ it depends only on $\pi_1(X)$; 
$HC^{[u^{-1}]}$ is the version for which the action of $u$ is nilpotent;
$HC^{[[u,u^{-1}]]}$ applied to a smooth $S^1$-space satisfies the fixed
point localization theorem of Jones and Petrack. 

(b) The notion of a mixed complex was introduced by
Kassel~\cite{Kassel87}. Its name reflects the fact that $\delta$ has degree $+1$
while $D$ has degree $-1$. Mixed complexes also appear
in~\cite{Getzler-Jones-Petrack} under the name {\em dg-$\Lambda$-module}.

(c) Note that the differential $\delta_u$ computing cyclic homology has
degree $+1$. This convention agrees with the one
in~\cite{Getzler-Jones-Petrack}, and it is related to the one
in~\cite{Kassel87,Jones,Loday} by replacing all degrees by their negatives.
More precisely, if $C_*$ is a mixed complex in our sense, then
$C_{-*}$ is a mixed complex in the sense of~\cite{Kassel87,Jones,Loday}
and their cyclic homology $H_*(C_{-*})$ equals our cyclic homology
$H_{-*}(C_*)$ in negative degrees. 
%Let us emphasize that our eight versions of $HC^*$ correspond to
%the cyclic {\em homology} of the mixed complex $(C,\delta,D)$. We
%write them with an upper $*$ because the differential $\delta_u$ has
%degree $+1$, in the same way that the homology of a cochain complex is
%denoted by $H^*$.
Our grading convention avoids unnecessary minus signs in our
main examples which arise from cochain complexes, such as the de Rham
complex of a manifold or the singular cochain complex of its free loop space. 
%To fit these examples into the framework of standard references for
%cyclic homology such as~\cite{Jones,Loday}, their degrees $*$ need to
%be replaced by $-*$ in order to view them as chain complexes. Their
%cyclic homology $HC_{-*}$ in the sense of~\cite{Jones,Loday} then
%equals our $HC^*$. 
\end{remark}

\begin{example}
Consider the mixed complex with $C_k:=\R$ in each degree $k\in\Z$ and
trivial differentials $d=D=0$. Then in each degree $k$ and for each
version $\{u,u^{-1}\}$ we have $HC_k^{\{u,u^{-1}\}}=C_k^{\{u,u^{-1}\}}=\R\{u,u^{-1}\}$, 
so all eight versions of cyclic homology are pairwise non-isomorphic as $\R[u]$-modules.
\end{example}

{\bf Quasi-isomorphism invariance. }
A morphism $f$ between mixed complexes $(C,\delta,D)$ and $(\wt C,\wt \delta,\wt D)$
induces homomorphisms $f_*$ between all versions of homology defined
above as modules over $\R[u]$ resp.~$\R[u,u^{-1}]$. 
Cleary $f_*$ is an isomorphism if $f$ is a homotopy equivalence. 
We say that a version of homology is {\em quasi-isomorphism invariant}
if the induced map $f_*$ is an isomorphism whenever $f$ is a
quasi-isomorphism of mixed complexes. 

\begin{prop}\label{prop:quism-invariance} 
The $3$ classical versions $HC^{[[u]]}$, $HC^{[[u,u^{-1}]}$ and
$HC^{[u^{-1}]}$ of cyclic homology are quasi-isomorphism
invariants of mixed complexes, whereas the other $5$ versions are not. 
\end{prop}

\begin{proof}
The quasi-isomorphism invariance of the $3$ classical versions is 
proved in~\cite[Lemma 2.1]{Jones} in the special case of cyclic chain
complexes, and in~\cite[Proposition 2.4]{Zhao} in the more general
context of $S^1$-complexes (cf.~Remark~\ref{rem:S1-complexes}
below). 
Examples~\ref{ex:point-deRham}, \ref{ex:point-sing} and~\ref{ex:ES1-sing} 
below in conjunction with Lemma~\ref{lem:examples-quiso} show that the
other $5$ homology groups are not quasi-isomorphism invariant.  
\end{proof}

\begin{remark}\label{rem:S1-complexes}
Much of the preceding discussion can be generalized to $S^1$-complexes
as defined in~\cite{Zhao}. This generalization is relevant if one
wants to include symplectic homology in this framework, but will not
be further discussed in this paper.
\end{remark}

{\bf Exact sequences. }
The eight versions of cyclic homology are connected by various exact
sequences fitting into commuting diagrams. 

\begin{prop}[Hood and Jones~\cite{Hood-Jones}]\label{prop:gysin}
For every mixed complex $(C,\delta,D)$ there exists a commuting
diagram with exact rows and columns
$$
\begin{CD}
   && \cdots && \cdots && \\
   && HC_{*+1}^{[u^{-1}]} @>{\id}>{=}> HC_{*+1}^{[u^{-1}]} && \\
   && @VV{D_0}V @VV{D_0}V && \\
   \cdots HC_{*-2}^{[u]} @>{\cdot u}>> HC_{*}^{[u]} @>{u=0}>> H_*(C,\delta) @>{D_*}>> HC_{*-1}^{[u]}\cdots \\
   @V{=}V{\id}V @VV{i_*}V @VV{i_*}V @V{=}V{\id}V \\
   \cdots HC_{*-2}^{[u]} @>{\cdot u}>> HC_{*}^{[u,u^{-1}]} @>{p_*}>> HC_*^{[u^{-1}]} @>{D_{0*}}>> HC_{*-1}^{[u]}\cdots \\
   && @VV{p_*\cdot u}V @VV{\cdot u}V && \\
   && HC_{*+2}^{[u^{-1}]} @>{\id}>{=}> HC_{*+2}^{[u^{-1}]} && \\
   && \cdots && \cdots && 
\end{CD}
$$
and similarly for the $[u,u^{-1}]]$, $[[u,u^{-1}]]$ and $[[u,u^{-1}]$ versions. 
Here $D_0$ means $D$ applied to the constant term in $u$ and the other
maps are the obvious ones.  
\end{prop}

We will refer to the second vertical and the first horizontal
sequences as the {\em Gysin (or Connes) exact sequences}, and to the first
vertical and the second horizontal sequences (which are equivalent in
view of the periodicity $HC_{*-2}^{[u,u^{-1}]}\cong HC_{*}^{[u,u^{-1}]}$) as
the {\em tautological exact sequences}.  

\begin{proof}
This diagram appears in~\cite[Figure 1]{Hood-Jones}. It follows from the commuting square of short exact sequences
$$
\begin{CD}
   && 0 && 0 && 0 && \\
   && @VVV @VVV @VVV && \\
   0 @>>> uC[u] @>>> C[u] @>>> C[u]/uC[u] @>>> 0 \\
   && @V{=}V{\id}V @VV{i}V @VV{i}V && \\
   0 @>>> uC[u] @>>> C[u,u^{-1}] @>{p}>> C[u,u^{-1}]/uC[u] @>>> 0 \\
   && @VVV @VV{p}V @VV{p}V && \\
   && 0 @>>> C[u,u^{-1}]/C[u] @>{\id}>{=}> C[u,u^{-1}]/C[u] @>>> 0 \\
   && && @VVV @VVV && \\
   && && 0 && 0 && &&
\end{CD}
$$
together with the identifications $C[u]/uC[u]=C$,
$C[u,u^{-1}]/uC[u]=C[u^{-1}]$, and $C[u,u^{-1}]/C[u]=u^{-1}C[u^{-1}]$.
\end{proof}

{\bf Other homologies. }
Given a mixed complex $(C,\delta,D)$, the chain complex $(C,\delta)$
has two natural subcomplexes $\im D\subset\ker D\subset C$ which
together with their quotient complexes fit into the commuting diagram
with exact rows and columns
$$
\xymatrix{
  & & & 0 \ar[d] \\
  & 0 \ar[d] & 0 \ar[d] & \ker D/\im D \ar[d] \\ 
  0 \ar[r] & \im D \ar[d] \ar[r] & (C,\delta) \ar@{=}[d] \ar[r] &
  C/\im D \ar[d] \ar[r] & 0 \\
  0 \ar[r] & \ker D \ar[d] \ar[r] & (C,\delta) \ar[d] \ar[r] & C/\ker
  D \ar[d] \ar[r] & 0 \\
  & \ker D/\im D \ar[d] & 0 & 0 \\
  & 0 
}
$$
On homology this yields the following commuting diagram
with exact rows and columns
$$
\xymatrix{
  H_{*-1}(\ker D/\im D) \ar[d]^{\delta_*} & & H_*(\ker D/\im D) \ar[d]
  \ar@{=}[r] & H_*(\ker D/\im D) \ar[d]^{\delta_*} \\  
  \cdots H_*(\im D) \ar[d] \ar[r] & H_*(C,\delta) \ar@{=}[d] \ar[r] &
  H_*(C/\im D) \ar[d] \ar[r]^{\delta_*} & H_{*+1}(\im D)\cdots \ar[d] \\
  \cdots H_*(\ker D) \ar[d] \ar[r] & H_*(C,\delta) \ar[r] & H_*(C/\ker
  D) \ar[d]^{\delta_*} \ar[r]^{\delta_*} & H_{*+1}(\ker D)\cdots \ar[d] \\
  H_*(\ker D/\im D) & & H_{*+1}(\ker D/\im D) \ar@{=}[r] & H_{*+1}(\ker D/\im D)
}
$$
A morphism $f$ of mixed complexes induces homomorphisms between all these
homologies, which are isomorphisms if $f$ is a homotopy equivalence. 
Note that the map $D_0$ in Proposition~\ref{prop:gysin} naturally
factors through chain maps (where $D$ has degree $-1$)
\begin{equation}\label{eq:D0}
  D_0:(C[[u^{-1}]],\delta_u)\stackrel{\pi_0}\longrightarrow
  (C/\im D,\delta)\stackrel{D}\longrightarrow
  (\im D,\delta)\stackrel{\iota_0}\longrightarrow
  (C[u],\delta_u).
\end{equation}

%%%
\subsection{Cocyclic and cyclic objects}\label{sec:cyc}
%%%

One source of mixed complexes are cyclic cochain complexes which we
introduce in this subsection. For more background see~\cite[Section 6.1]{Loday}. 

{\bf Cocyclic objects. }
A {\em cocyclic object} in some category is a sequence of objects
$C_n$, $n\in\N_0$, with morphisms 
\begin{itemize}
\item $\delta_i:C_{n-1}\to C_n$, $i=0,\dots,n$ (faces),
\item $\sigma_j:C_{n+1}\to C_n$, $j=0,\dots,n$ (degeneracies),
\item $\tau_n:C_n\to C_n$ (cyclic operators)   
\end{itemize}
satisfying the following relations:
\begin{align*}
   \delta_j\delta_i &= \delta_i\delta_{j-1}\quad\text{for }i<j,\cr
   \sigma_j\sigma_i &= \sigma_i\sigma_{j+1}\quad\text{for }i\leq j,\cr
   \sigma_j\delta_i &= \begin{cases}
      \delta_i\sigma_{j-1} & \text{for }i<j,\cr
      \id & \text{for }i=j\text{ or }i=j+1,\cr 
      \delta_{i-1}\sigma_j & \text{for }i>j+1,\cr
   \end{cases} \cr
   \tau_n\delta_i &= \delta_{i-1}\tau_{n-1}\quad\text{for }1\leq i\leq
   n,\qquad \tau_n\delta_0 = \delta_n,\cr
   \tau_n\sigma_i &= \sigma_{i-1}\tau_{n+1}\quad\text{for }1\leq i\leq
   n,\qquad \tau_n\sigma_0 = \sigma_n\tau_{n+1}^2,\cr
   \tau_n^{n+1} &= \id\,.
%   &\text{and relations involving the }\sigma_i
\end{align*}
Forgetting $\tau_n$ we have a {\em cosimplicial object}, and
forgetting the $\sigma_j$ a {\em pre-cocyclic object}. 

\begin{example}\label{ex:top}
A topological space $X$ gives rise to a cocyclic space by
setting $C_n(X):=X^{n+1}=X\times\cdots\times X$ ($n+1$ times) and
\begin{align*}
   \delta_i(x_0,\dots,x_{n-1}) &:= (x_0,\dots,x_i,x_i,\dots,x_{n-1})
   \quad\text{for }0\leq i\leq n-1,\cr
   \delta_n(x_0,\dots,x_{n-1}) &:= (x_0,\dots,x_{n-1},x_0),\cr
   \sigma_j(x_0,\dots,x_{n+1}) &:= (x_0,\dots,\wh{x_i},\dots,x_{n+1}),\cr
   \tau_n(x_0,\dots,x_n) &:= (x_1,\dots,x_n,x_0).
\end{align*}
For a subspace $Y\subset X$ we can define a cosimplicial space by 
$C_n(X,Y):= X^n\times Y$ and the same operations $\delta_i,\sigma_j$
(appropriately viewing elements in $Y$ as elements in $X$ via the inclusion).  
\end{example}

{\bf Cyclic objects. }
Dualizing the notion of a cocyclic object, we obtain that of a cyclic
object. 
A {\em cyclic object} in some category is a sequence of objects
$C_n$, $n\in\N_0$, with morphisms 
\begin{itemize}
\item $d_i:C_n\to C_{n-1}$, $i=0,\dots,n$ (faces),
\item $s_j:C_n\to C_{n+1}$, $j=0,\dots,n$ (degeneracies),
\item $t_n:C_n\to C_n$ (cyclic operators)   
\end{itemize}
satisfying the following relations:
\begin{align*}
   d_id_j &= d_{j-1}d_i \quad\text{for }i<j,\cr
   s_is_j &= s_{j+1}s_i\quad\text{for }i\leq j,\cr
   d_is_j &= \begin{cases}
      s_{j-1}d_i & \text{for }i<j,\cr
      \id & \text{for }i=j\text{ or }i=j+1,\cr 
      s_jd_{i-1} & \text{for }i>j+1,\cr
   \end{cases} \cr
   d_it_n &= t_{n-1}d_{i-1}\quad\text{for }1\leq i\leq n,\qquad d_0t_n=d_n, \cr
   s_it_n &= t_{n+1}s_{i-1}\quad\text{for }1\leq i\leq n,\qquad s_0t_n=t_{n+1}^2s_n,\cr
   t_n^{n+1} &= \id\,.
%   &\text{and relations involving the }s_i
\end{align*}
Forgetting $t_n$ we have a {\em simplicial object}, and
forgetting the $s_j$ a {\em pre-cyclic object}. 
Note that if $F:A\to B$ is a contravariant functor, then a cocyclic
object in the category $A$ gives rise to a cyclic object $F(A)$ in the
category $B$ and vice versa. 

The cyclic structure gives rise to the {\em extra degeneracy}
$$
   s_{n+1}:=t_{n+1}s_n:C_n\to C_{n+1}
$$
satisfying the following relations on $C_n$: %(stated incorrectly in~\cite{Loday}):
\begin{align*}
   s_is_{n+1} &= s_{n+2}s_{i-1}\quad\text{for }1\leq i\leq n+1,\qquad s_0s_{n+1}=s_{n+2}s_{n+1},\cr
   d_is_{n+1} &= s_nd_{i-1}\quad\text{for }1\leq i\leq n,\qquad
   d_{n+1}s_{n+1}=t_n,\qquad d_0s_{n+1}=\id, \cr
   s_{n+1}t_n &= t_{n+1}^2s_{n-1}\,.
\end{align*}

\begin{example}\label{ex:forms}
For a manifold $X$ denote by $C_n(\Om^*(X)):=\Om^*(X^{n+1})$ the differential
graded algebra of differential forms on the $(n+1)$-fold product of
$X$. This becomes a cyclic object in the category of differential
graded vector spaces (i.e., chain complexes) with the
operations $d_i:=\delta_i^*$, $s_j:=\sigma_j^*$, $t_n:=\tau_n^*$
induced from those of Example~\ref{ex:top}.
For a submanifold $Y\subset X$ we can define a simplicial chain
complex by $C_n(\Om^*(X),\Om^*(Y)):= \Om^*(X^n\times Y)$ and the same operations $d_i,s_j$
(appropriately viewing forms on $X$ as forms on $Y$ via restriction). Note
that $\Om^*(Y)$ is an $\Om^*(X)$-bimodule. 
\end{example}

\begin{example}\label{ex:algebras}
A differential graded algebra ({\em dga}) $A$ over $\R$ with unit $1$ gives rise to a
cyclic object in the category of differential graded $\R$-modules
by setting $C_n(A):=A^{\otimes(n+1)}$ and 
\begin{align*}
   d_i(a_0\mid\dots\mid a_n) &:= (a_0\mid\dots\mid
   a_ia_{i+1}\mid\dots\mid a_n)\quad\text{for }0\leq i\leq n-1,\cr
   d_n(a_0\mid\dots\mid a_n) &:= (-1)^{\deg a_n(\deg a_0+\dots+\deg a_{n-1})}
   (a_na_0\mid a_1\mid\dots\mid a_{n-1}),\cr 
   s_j(a_0\mid\dots\mid a_n) &:= (a_0\mid\dots\mid a_j\mid 1\mid
   a_{j+1}\mid\dots\mid a_n)\quad\text{for }0\leq i\leq n,\cr
   t_n(a_0\mid \dots\mid a_n) &:= (-1)^{\deg a_n(\deg a_0+\dots+\deg a_{n-1})}
   (a_n\mid a_0\mid\dots\mid a_{n-1}). 
\end{align*}
Here for $t_n$ we use the convention in~\cite{Loday}, which differs
the one in~\cite{Jones}. 
The extra degeneracy becomes 
$$
   s_{n+1}(a_0\mid\dots\mid a_n) = (1\mid a_0\mid\dots\mid a_n).
$$
For a differential graded $A$-bimodule $M$ we can define a simplicial
differential graded $R$-module by $C_n(A,M):= A^{\otimes n}\otimes
M$ and the same operations $d_i,s_j$ (using the bimodule structure).  
\end{example}

For the de Rham complex, the constructions in Examples~\ref{ex:forms}
and~\ref{ex:algebras} give rise to two cyclic cochain complexes
$\Om^*(X^{n+1})$ and $\Om^*(X)^{\otimes n+1}$ which we will refer to
as the {\em analytic} and {\em algebraic} versions, respectively.
They are compatible in the following sense. 

\begin{lemma}\label{lem:cross-product}
The exterior cross products
$$
   \phi_n:\Om^*(X)^{\otimes n+1}\to\Om^*(X^{n+1}),\qquad (a_0\mid\dots\mid
   a_n)\mapsto a_0\times\cdots\times a_n
$$
define a morphism of cyclic cochain complexes. 
\end{lemma}

\begin{proof}
Each $\phi_n$ is clearly a chain map with respect to the exterior
derivative. Recall that 
$$
   a_0\times\cdots\times a_n = \pi_0^*a_0\wedge\cdots\wedge\pi_n^*a_n
$$
for the canonical projections $\pi_i:X_n\to X$, $i=0,\dots,n$. Note
that the $\pi_i$ and $\pi_i^*$ are compositions of the degeneracies in
Examples~\ref{ex:top} and~\ref{ex:forms}, respectively. Using this,
one deduces that $\phi$ is a map of cyclic cochain complexes. For
example, the relations
$$
   \pi_{i-1}\tau_n=\pi_i\quad\text{for }1\leq i\leq n,\quad
   \pi_n\tau_n=\pi_0 
$$
imply compatibility of $\phi$ with $t_n=\tau_n^*$:
\begin{align*}
   \phi\,t_n(a_0\mid\dots\mid a_n)
   &= (-1)^{\deg a_n(\deg a_0+\dots+\deg a_{n-1})}a_n\times
   a_0\times\cdots\times a_{n-1} \cr 
   &= (-1)^{\deg a_n(\deg a_0+\dots+\deg a_{n-1})}\pi_0^*a_n\wedge
   \pi_1^*a_0\wedge\cdots\wedge \pi_n^*a_{n-1} \cr 
   &= (-1)^{\deg a_n(\deg a_0+\dots+\deg a_{n-1})}\tau_n^*(\pi_n^*a_n\wedge
   \pi_0^*a_0\wedge\cdots\wedge \pi_{n-1}^*a_{n-1}) \cr 
   &= \tau_n^*(\pi_0^*a_0\wedge\cdots\wedge \pi_n^*a_n) \cr 
   &= t_n\phi(a_0\mid\dots\mid a_n).
\end{align*}
\end{proof}

%%%
\subsection{From cyclic cochain complexes to mixed complexes}
%%%

Consider a cyclic cochain complex $(C_n,d_i,s_j,t_n,d)$, where
$(C_n,d)$ are $\Z$-graded cochain complexes for $n\in\N_0$ with
differential of degree $+1$. Thus the operations $d_i$, $s_j$ and
$t_n$ commute with $d$ and satisfy the relations in
Section~\ref{sec:cyc}. 
We define a $\Z$-graded $\R$-vector space $(C,|\cdot|)$ by
$$
 C:=\bigoplus_{n\geq 0}C_n,\qquad |c|:=\deg(c)-n\quad\text{for }c\in C_n.
$$
We define the following operations on homogeneous elements $c\in C_n$:
\begin{equation}\label{eq:cyc-signs}
\begin{aligned}
   b'(c)&:=(-1)^{|c|+1}\sum_{i=0}^{n-1}(-1)^id_i(c)\in C_{n-1}, \cr 
   b(c) &:=(-1)^{|c|+1}\sum_{i=0}^n(-1)^id_i(c)\in C_{n-1}, \cr 
   t &:= (-1)^nt_n:C_n\to C_n, \cr
   N &:= 1+t+\cdots+t^n:C_n\to C_n, \cr
   s(c) &:= (-1)^{|c|}s_{n+1}(c)\in C_{n+1}, \cr
   B &:= (1-t)sN:C_n\to C_{n+1}. 
\end{aligned}
\end{equation}
The operations have the following degrees with respect to the grading
on $C$:
$$
   |d|=|b|=|b'|=1,\qquad |t|=|N|=0,\qquad |s|=|B|=-1. 
$$
Straightforward computations yield the following relations:
\begin{equation}\label{eq:commute}
\begin{gathered}
   \ker(1-t)=\im N,\qquad \ker N=\im(1-t),\qquad (1-t)b'=b(1-t), \cr b'N=Nb,\qquad
   (b')^2=b^2=db+bd=db'+b'd=0, \cr
   ds+sd=0, \qquad b's+sb'=1, \qquad bs+sb'=(1-t).
\end{gathered}
\end{equation}

\begin{remark}
The signs in~\eqref{eq:cyc-signs} are chosen for compatibility under
Chen's iterated integral in Section~\ref{sec:chen}. An alternative
sign convention is given in~\cite{Loday} (in the case of a cyclic module):
\begin{equation}\label{eq:cyc-signs-Loday}
\begin{aligned}
   \wt b'&:=\sum_{i=0}^{n-1}(-1)^id_i:C_n\to C_{n-1}, \cr 
   \wt b&:=\sum_{i=0}^n(-1)^id_i:C_n\to C_{n-1}, \cr 
   \wt t &:= (-1)^nt_n:C_n\to C_n, \cr
   \wt N &:= 1+t+\cdots+t^n:C_n\to C_n, \cr
   \wt s &:=s_{n+1}:C_n\to C_{n+1}, \cr
   \wt B &:= (1-\wt t)\wt s\wt N:C_n\to C_{n+1}. 
\end{aligned}
\end{equation}
Note that $\wt t=t$ and $\wt N=N$ are unchanged. In this case we need
to modify the differential $d$ to make it anticommute with $\wt b$,
$\wt b'$ and $\wt s$, so we set
%\marginpar{Still to be checked (Kai).}
$$
   \wt d(c):=(-1)^{|c|+1}d(c). 
$$
It is straightforward to check that with these signs the
relations~\eqref{eq:commute} continue to hold. Moreover, the two sign
conventions are intertwined by the chain isomorphism 
$\Phi:(C,\wt d)\longrightarrow (C,d)$ defined on homogeneous elements by
$\Phi(c) := (-1)^{|c|(|c|+1)/2}c$.
%$$
%   \Phi(c) := (-1)^{x(c)}c,\qquad x(c):=\frac{|c|(|c|+1)}{2}.
%$$
\end{remark}

Consider the space $C[[\theta,\theta^{-1}]]$ of Laurent series in a formal
variable $\theta$ of degree $|\theta|=1$, as usual understood in the graded
sense so that elements of degree $k$ are sums
$\sum_{i\in\Z}c_i\theta^i$ with $|c_i|=k-i$. We view this as a double
complex where powers of $\theta$ increase in the horizontal direction
and the columns are copies of $C$, see Figure~\ref{fig1}. We define a differential 
$$
   \delta_\theta:=\delta_\ver+\delta_\hor:C[[\theta,\theta^{-1}]]\to C[[\theta,\theta^{-1}]]
$$
of degree $1$ with
\begin{equation}\label{eq:delta-theta}
\begin{aligned}
   \delta_\ver(c\theta^n) &:= \begin{cases}
      (d+b)(c)\theta^n & \text{$n$ even}, \\
      -(d+b')(c)\theta^n & \text{$n$ odd}, \\
   \end{cases} \cr
   \delta_\hor(c\theta^n) &:= \begin{cases}
      N(c)\theta^{n+1} & \text{$n$ even}, \\
      (1-t)(c)\theta^{n+1} & \text{$n$ odd}. \\
   \end{cases}
\end{aligned}
\end{equation}

 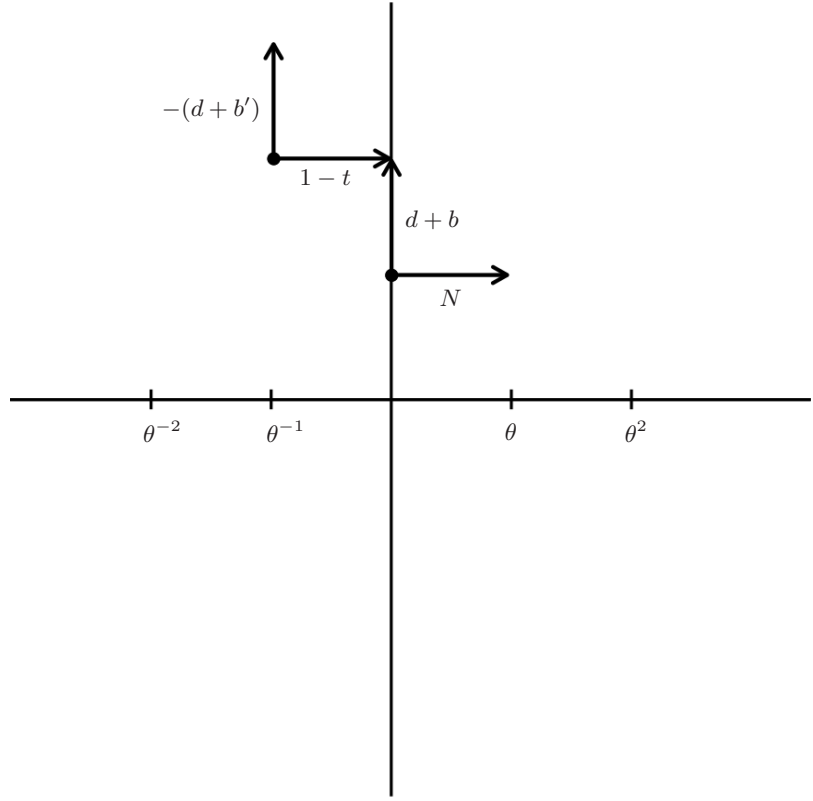
\begin{figure}
\labellist
\small
\pinlabel* $\theta^{-2}$ at 100 193 
\pinlabel* $\theta^{-1}$ at 153 193 
\pinlabel* $\theta$ at 250 193 
\pinlabel* $\theta^{2}$ at 304 193 
\pinlabel* $N$ at 224 251
\pinlabel* $d+b$ at 216 285 
\pinlabel* $1-t$ at 170 303 
\pinlabel* $-(d+b')$ at 121 333 
\endlabellist 
\centering
  \includegraphics[width=1.0\textwidth]{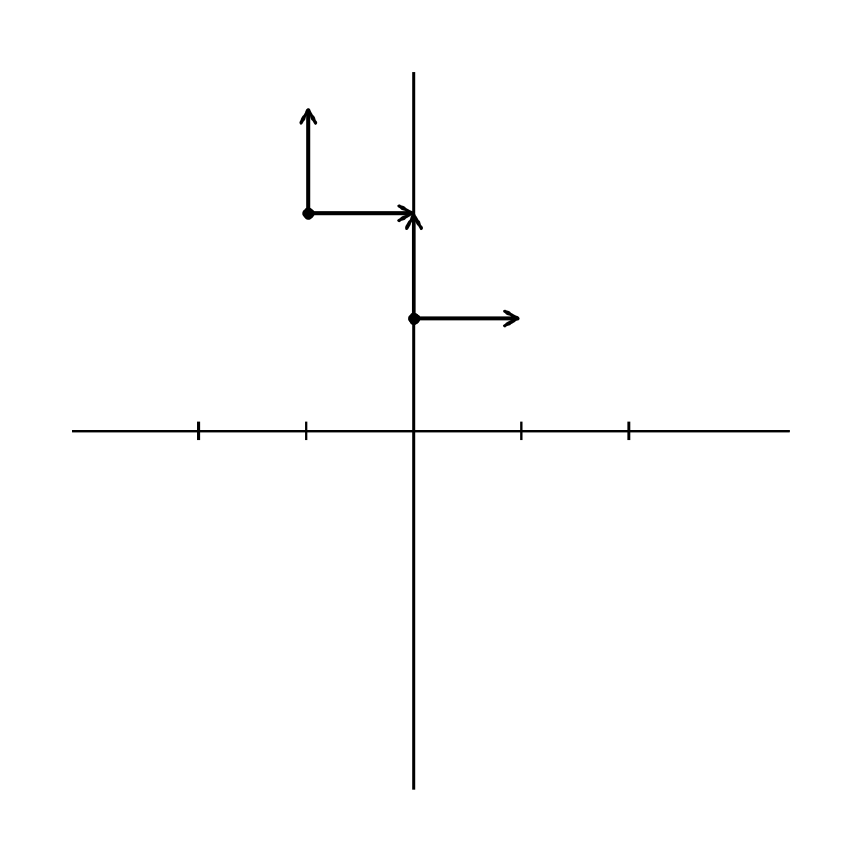}
  \caption{The double complex $C[[\theta,\theta^{-1}]]$}
  \label{fig1}
\end{figure}

The relations~\eqref{eq:commute} imply that this is indeed a double complex.
Its total differential $\delta_\theta$ induces differentials on all
the seven sub- and quotient complexes $C[[\theta,\theta^{-1}]$ etc,
defined as in~\eqref{eq:cyc-versions} with $u$ replaced by $\theta$. 
The double complex $C[[\theta,\theta^{-1}]]$ has the following two key
properties:
\begin{enumerate}
\item its odd columns are {\em contractible} with contracting homotopy $-s$, i.e., $(d+b')s+s(d+b')=\id$;
\item its rows are exact (cf.~\cite{Loday} Theorem 2.1.5).
\end{enumerate}

\begin{lemma}\label{lem:cyccomplex-dgmod}
Let $(C_n,d_i,s_j,t_n,d)$ be a cyclic cochain complex. Then 
$$
   (C,d_H:=d+b,B)
$$ 
is a mixed complex. Moreover, the chain map  
$$
   k:C[[u,u^{-1}]]\longrightarrow C[[\theta,\theta^{-1}]],\qquad
   \sum_jc_ju^j\mapsto \sum_j(c_j\theta^{2j}+sNc_j\theta^{2j+1}) 
$$
induces an isomorphism on homology as an 
$\R[u,u^{-1}]$-module, where $u$ acts on $C[[\theta,\theta^{-1}]]$ as multiplication by
$\theta^2$. Similarly for all the other versions of homology. 
Moreover, these isomorphisms fit into a commuting diagram
\begin{equation}\label{eq:u-theta}
\xymatrix{
   \cdots HC_{*-2}^{[u]} \ar[d]_\cong^{k_*} \ar[r]^{\cdot u} &
   HC_*^{[u,u^{-1}]} \ar[d]_\cong^{k_*} \ar[r]^{p_*} & HC_*^{[u^{-1}]}
   \ar[d]_\cong^{k_*} \ar[r]^{B_0} & HC_{*-1}^{[u]}\cdots \ar[d]_\cong^{k_*}\\
   \cdots HC_{*-2}^{[\theta]} \ar[r]^{\cdot\theta^2} &
   HC_*^{[\theta,\theta^{-1}]} \ar[r]^{r_*p_*} & HC_*^{[\theta^{-1}]}
   \ar[r]^{B_0} & HC_{*-1}^{[\theta]} \cdots\,
}
\end{equation}
and similarly for all other tautological and Connes exact sequences.
\end{lemma}

\begin{proof}
The proof is an easy adaptation of the arguments in~\cite[Section
  2.1]{Loday}: that $d+b$ and $B$ are anticommuting differentials and $k$ is a
chain map follows by direct computation based on~\eqref{eq:commute},
and that $k$ induces an isomorphism on homology follows from
contractibility of the odd columns of the double complex
$C[[\theta,\theta^{-1}]]$ and~\cite[Lemma~2.1.6]{Loday}. 
%As a sample computation:
%$$
%bB=b(1-t)sN=(1-t)b'sN=(1-t)(-sb'+1)N=-(1-t)sb'N
%$$
%$$
%=-(1-t)sb'N=-(1-t)sNb=-Nb.
%$$
To derive the diagram~\eqref{eq:u-theta}, consider the commuting
diagram of short exact sequences 
$$
\xymatrix{
   0 \ar[r] & uC[u] \ar[d]^k \ar[r] & C[u,u^{-1}] \ar[d]^k \ar[r]^p &
   C[u^{-1}] \ar[d]^q \ar[r] & 0 \\
   0 \ar[r] & \theta^2C[\theta] \ar[r] & C[\theta,\theta^{-1}] \ar[r]^p &
   \theta C[\theta^{-1}] \ar[d]^r \ar[r] & 0 \\
   & & & C[\theta^{-1}],
}
$$
where $q$ is the induced map between the quotients and $r$ is the map
dividing out the $\theta$-column. Since that column is
contractible, the induced map $r_*$ on homology is an isomorphism and
we obtain on homology the diagram
$$
\xymatrix{
   \cdots HC_{*-2}^{[u]} \ar[d]_\cong^{k_*} \ar[r]^{\cdot u} &
   HC_*^{[u,u^{-1}]} \ar[d]_\cong^{k_*} \ar[r]^{p_*} & HC_{*}^{[u^{-1}]}
   \ar[d]_\cong^{q_*} \ar[r]^{B_0} & HC_{*-1}^{[u]}\cdots \ar[d]_\cong^{k_*}\\
   \cdots HC_{*-2}^{[\theta]} \ar[r]^{\cdot\theta^2} & HC_*^{[\theta,\theta^{-1}]}
   \ar[r]^{p_*} & H_*(\theta C[\theta^{-1}]) \ar[d]_\cong^{r_*}
   \ar[r] & HC_{*-1}^{[\theta]}\cdots \\
   & & HC_{*}^{[\theta^{-1}]} \ar[ru]_{B_0}
}
$$
Using $r_*q_*=k_*$ we obtain~\eqref{eq:u-theta}. 
\end{proof}

%The key advantage of the double complex $C[[\theta,\theta^{-1}]]$ is that its rows are exact.

The homology $H(C,d_H:=d+b)$ is the {\em Hochschild homology}
and the homologies $HC^{[[u,u^{-1}]]}$ etc are the various
versions of {\em cyclic homology} of the cyclic cochain complex. 

{\bf Connes' version of cyclic homology. }
Due to the relations~\eqref{eq:commute} we have short exact sequences
of cochain complexes 
\begin{gather*}
   0 \longrightarrow \bigl(\im(1-t),d+b\bigr) \longrightarrow (C,d+b) \longrightarrow
   \bigl(C/\im(1-t),d+b\bigr) \longrightarrow 0, \cr
   0 \longrightarrow (\im N,d+b') \longrightarrow (C,d+b') \longrightarrow
   (C/\im N,d+b') \longrightarrow 0. 
\end{gather*}
The first one induces a long exact sequence
$$
   \cdots H_*\bigl(\im(1-t),d+b\bigr) \longrightarrow H_*(C,d+b) \longrightarrow
   H_*\bigl(C/\im(1-t),d+b\bigr) \longrightarrow H_{*+1}\bigl(\im(1-t),d+b\bigr) \cdots
$$
and the second one, by acyclicity of $(C,d+b')$, an isomorphism
$$
   H_*(C/\im N,d+b')\cong H_{*+1}(\im N,d+b').
$$ 
Moreover, the chain isomorphisms
$N:\bigl(C/\im(1-t),d+b\bigr)\stackrel{\cong}\longrightarrow (\im
N,d+b')$ and
$1-t:(C/\im N,d+b')\stackrel{\cong}\longrightarrow \bigl(\im(1-t),d+b\bigr)$
induce isomorphisms 
\begin{gather*}
   N:H_*\bigl(C/\im(1-t),d+b\bigr)\stackrel{\cong}\longrightarrow H_*(\im N,d+b'),\cr
   1-t:H_*(C/\im N,d+b')\stackrel{\cong}\longrightarrow H_*\bigl(\im(1-t),d+b\bigr).
\end{gather*}
\begin{definition}
{\em Connes' version of cyclic homology} of the cyclic cochain complex
$(C_n,d_i,s_j,t_n,d)$ is
\begin{align*}
   HC^\lambda_* &:= H_{*+1}\bigl(C/\im(1-t),d+b\bigr)\cong H_{*+1}(\im N,d+b')\cr
   &\ \cong H_{*}(C/\im N,d+b') \cong H_{*}\bigl(\im(1-t),d+b\bigr).
\end{align*}
\end{definition}

The following lemma identifies Connes' version of cyclic homology
with two of the other versions. 

\begin{lemma}\label{lem:cyccomplex-computation}
Let $(C_n,d_i,s_j,t_n,d)$ be a cyclic cochain complex.
Then $HC_*^{[u,u^{-1}]]}=HC_*^{[\theta,\theta^{-1}]]}=0$ and we have a commuting diagram 
\begin{equation}\label{eq:cyclic-cochain}
\xymatrix{
   HC_*^{[[u^{-1}]]} \ar[d]_{\cong}^{k_*} \ar[r]^{B_0}_{\cong} & HC_{*-1}^{[u]} \ar[d]_{\cong}^{k_*}\\
   HC_*^{[[\theta^{-1}]]} \ar[d]_{\cong}^{e_*} \ar[r]^{B_0}_{\cong} & HC_{*-1}^{[\theta]} \\
   H_*\bigl(C/\im(1-t),d+b\bigr) \ar[r]^{N_*}_{\cong} & H_{*}(\im
   N,d+b') \ar[u]^{\cong}_{((1-t)s)_*}\,
}
\end{equation}
where the upper square arises from~\eqref{eq:u-theta} and
$e:C[[\theta^{-1}]]\to C/\im(1-t)$ is the projection onto the $0$-th column.
\end{lemma}

\begin{proof}
Consider the double complex with exact rows $(C[\theta,\theta^{-1}]],\delta_\theta)$. 
Standard zigzag arguments as in~\cite{Loday} (see also the proof of
Lemma~\ref{lem:our-iso} below) show that its total homology vanishes
and the map $e:C[[\theta^{-1}]]\to C/\im(1-t)$ induces an isomorphism
on homology. In view of~\eqref{eq:u-theta} this yields the
diagram~\eqref{eq:cyclic-cochain}, where commutativity of the lower
square follows from the definition of $B$. The maps $B_0$ are isomorphism
because the third terms in the corresponding tautological exact sequences vanish. Since $N_*$
is an isomorphism, this implies that the map $((1-t)s)_*$ is an isomorphism as well.
It is important to note, however, that these arguments fail for other
versions of cyclic homology. 
\end{proof}

The following refinement of Lemma~\ref{lem:cyccomplex-computation}
will be crucial in the following. 

\begin{lemma}\label{lem:our-iso}
Let $(C_n,d_i,s_j,t_n,d)$ be a cyclic cochain complex.  
%such that for each $n$ the degrees in $C_n$ are bounded from above. 
Then the canonical inclusion $\phi:C[\theta^{-1}]\to C[[\theta^{-1}]]$
induces an isomorphism on homology
\begin{equation*}
   \phi_*:HC_*^{[\theta^{-1}]}\stackrel{\cong}\longrightarrow HC_*^{[[\theta^{-1}]]}. 
\end{equation*}
The same holds true for the inclusions $C[u^{-1}]\to C[[u^{-1}]]$,
$C[u,u^{-1}]\to C[u,u^{-1}]]$, and $C[[u,u^{-1}]\to C[[u,u^{-1}]]$.
\end{lemma}

\begin{proof}
Consider the sequence of chain maps
\begin{equation}
\begin{CD} 
   (C[\theta^{-1}],\delta_\theta) @>{\phi}>> (C[[\theta^{-1}]],\delta_\theta) @>e>> (C/\im(1-t),d+b).
\end{CD}
\end{equation}
%\begin{equation}
%\begin{CD} 
%   (C[\theta^{-1}],\delta_\theta) @>{\phi}>> (C[[\theta^{-1}]],\delta_\theta) @>p>> (C/\im(1-t),d+b) \\     
%   @AAA @AAA\\            
%   (C[u^{-1}],\delta_u) @>>> (C[[u^{-1}]],\delta_u) \;. &         
%\end{CD}
%\end{equation}
Since $e$ induces an isomorphism on homology by
Lemma~\ref{lem:cyccomplex-computation}, it suffices to show 
that the composition $e\phi$ induces an isomorphism $(e\phi)_*$ on homology. 
%The proof consists in a closer inspection of
%resolutions used to prove that $p$ induces an isomorphism on homology in Lemma~\ref{lem:gen-zigzag}.

We call an element $c\in\theta^{-1}C[[\theta^{-1}]]$ a {\em
  resolution} of an element $x\in C$ if $\delta_\theta c=x$ (in
particular, all nonzero powers of $\theta^{-1}$ cancel). Since the
rows of $C[[\theta^{-1}]]$ are exact, an element $x\in C$ admits a
resolution if and only if $(d+b)x=0$ and $x\in\im(1-t)=\ker N$. 

{\em Claim. }Every $x\in C$ with $(d+b)x=0$ and $x\in\im(1-t)$ admits
a resolution $c$ which is {\em polynomial} in $\theta^{-1}$,
i.e.~$c\in \theta^{-1}C[\theta^{-1}]$. 

Let us assume the claim for the moment and finish the proof. 
We begin with surjectivity of $(e\phi)_*$. 
Let $y\in C$ represent a closed element in $(C/\im(1-t),d+b)$. 
Then $(d+b)y$ is $(d+b)$-closed and lies in $\im(1-t)$, so by the
claim it has a polynomial resolution $c\in\theta^{-1}C[\theta^{-1}]$. 
Then 
$$
   \delta_\theta(y-c)=Ny\theta + (d+b)y - \delta_\theta c =
   Ny\theta\equiv 0\in C[\theta^{-1}]=C[\theta,\theta^{-1}/\theta C[\theta],
$$
so the element $y-c$ is closed in $C[\theta^{-1}]$. Since $e\phi
(y-c)=y$, this proves surjectivity of $(e\phi)_*$. 

Next we show that $(e\phi)_*$ is injective. Since $e_*$ is an
isomorphism, it suffices to show that $\phi_*$ is injective. 
Let a closed element $c=\sum_{k=0}^nc_{-k}\theta^{-k}\in C[\theta^{-1}]$ represent an 
exact element in $C[[\theta^{-1}]]$, i.e., $\delta_\theta a=c$ for
some $a = \sum_{k=0}^\infty a_{-k}\theta^{-k}\in C[[\theta^{-1}]]$. 
We may assume that $n$ is odd by allowing the last coefficient $c_{-n}$ to be zero.
Then $(d+b)a_{-n-1}$ is $(d+b)$-closed and lies in $\im(1-t)$, so by the
claim it has a polynomial resolution $z = \sum_{k=1}^Nz_{-k}\theta^{-1}$. 
Now
$$
   \tilde a := \sum_{k=0}^{n+1}a_{-k}\theta^{-k} - z\theta^{-n-1}
%   = \sum_{k=0}^{n+1}a_{-k}\theta^{-k} - \sum_{k=1}^Nz_{-k}\theta^{-n-1-k} 
   \in C[\theta^{-1}]
$$
satisfies 
$$
   \delta_\theta\tilde a = c + (d+b)a_{-n-1}\theta^{-n-1} -
   \delta_\theta z\,\theta^{-n-1} = c,
$$
so $c$ is exact in $C[\theta^{-1}]$ and injectivity is proved.

It remains to prove the claim. 
By a slight abuse of notation, we will denote by $\delta_\ver$ and
$\delta_\hor$ the maps $C\to C$ defined as in~\eqref{eq:delta-theta}
but without the factors of $\theta$. 
Recall the formula $c=\Sigma_{k=1}^\infty c_{-k}\theta^{-k}$ for a
resolution of $x\in C$. We will write 
$$
   c_{-1} = \delta_\hor^{-1}(x)\qquad\text{and}\qquad c_{-(k+1)} =
   -\delta_\hor^{-1}\delta_\ver c_{-k}\,\, \text{for}\,\, k\ge 1, 
$$ 
where by $\delta_\hor^{-1}$ we mean some preimage under
$\delta_\hor$. The main point of the proof is a certain special choice
of the preimage. To explain it, recall that $C=\bigoplus_{n=0}^\infty
C_n$. The maps $t,N,d$ preserve $C_n$, whereas $b$ and $b'$ map
$C_{n+1}$ to $C_n$. In particular, $\delta_\hor$ preserves $C_n$. 
%This allows us to stipulate that  
%$$
%\delta_{hor}^{-1}\left(\bigoplus_{j\in I}C_j\right)\subset \bigoplus_{j\in I}C_j
%$$
%for any finite subset $I$ of $\N_0$. In other words, when we take
%preimage of $c$, we do not make it unnecessary large, by adding 
%closed elements in irrelevant summands $C_j$. 
Let us define the {\em weight} $w(c)$ of a nonzero element $c\in C$ as the
smallest integer $w$ such that $c\in\bigoplus_{n=0}^wC_n$, and set
$w(0):=-1$. By the preceding discussion we can choose the preimages
under $\delta_\hor$ such that 
\begin{equation}\label{eq:nonstrict}
   w(c_{-k-1})\le w(c_{-k})\qquad\text{for all }k\geq 1.
\end{equation}
This condition can be improved further. For this, note that 
the map $N:C_n\rightarrow C_n$ acts on its image as multiplication  
by $1+n$. Therefore $N^{-1}$ on the image of $N$ can be taken to be
multiplication by $(1+n)^{-1}$. In particular, it maps $d$-closed
elements to $d$-closed ones. From this we can deduce the following
strict (!) inequality: 
\begin{equation}\label{eq:strict}
   w(c_{-k-2})< w(c_{-k})\quad \text{for $k$ odd and }c_{-k}\ne 0.
\end{equation}
Indeed, for $k$ odd we obtain from~\eqref{eq:nonstrict}
$$
   w(c_{-k-2}) = w(\delta_\hor^{-1}\delta_\ver\delta_\hor^{-1}\delta_\ver c_{-k})
   \le w(\delta_\ver\delta_\hor^{-1}\delta_\ver c_{-k}) 
   = w\bigl((d+b)N^{-1}(d+b')c_{-k}\bigr). %< w(c_{-k})
$$
To compute the last term we write out
$$
   (d+b)N^{-1}(d+b')c_{-k}=dN^{-1}dc_{-k}+dN^{-1}b'c_{-k}+bN^{-1}dc_{-k}+bN^{-1}b'c_{-k},
$$
where the first summand $dN^{-1}dc_k=0$ vanishes since $N^{-1}$ maps
closed elements to closed ones. All other summands have weight
strictly less than that of $c_{-k}$ due to the presence of either $b$ or
$b'$ (or both), and inequality~\eqref{eq:strict} follows.  

Inequalities~\eqref{eq:nonstrict} and~\eqref{eq:strict} imply that
the weight $w(c_{-k})$ strictly decreases each time $k$ increases by $2$,
hence $c_{-k}=0$ for all sufficiently large $k$ and the resolution is polynomial.  
This proves the claim and thus the isomorphism
$\phi_*:HC_*^{[\theta^{-1}]}\stackrel{\cong}\longrightarrow HC_*^{[[\theta^{-1}]]}$.  
The statement about the inclusion $C[u^{-1}]\to C[[u^{-1}]]$ follows
from this and Lemma~\ref{lem:cyccomplex-dgmod}, and the statement
about the other two inclusions follows from the tautological exact
sequence in Proposition~\ref{prop:gysin} and the 5-lemma. 
\end{proof}

The preceding two lemmas are summarized in

\begin{cor}\label{cor:cyccochain}
For a cyclic cochain complex $(C_n,d_i,s_j,t_n,d)$ the different
versions of cyclic homology are given by
\begin{gather*}
   HC_*^{[u^{-1}]}\cong HC_*^{[[u^{-1}]]}\cong HC_{*-1}^\lambda\cong HC_{*-1}^{[u]},\qquad
   HC_*^{[[u]]}, \cr
   HC_*^{[[u,u^{-1}]}\cong HC_*^{[[u,u^{-1}]]},\qquad
   HC_*^{[u,u^{-1}]} = HC_*^{[u,u^{-1}]]} = 0.
\end{gather*}
\end{cor} 

{\bf Behaviour under maps of cyclic cochain complexes. }
According to Proposition~\ref{prop:quism-cyclicchains}, only the $3$
classical versions of cyclic homology are invariant under
quasi-isomorphisms of mixed complexes. By contrast, with respect to 
quasi-isomorphisms of cyclic cochain complexes we have
 
\begin{prop}\label{prop:quism-cyclicchains}
Hochschild homology and all versions of cyclic homology are invariant
under quasi-isomorphisms of cyclic cochain complexes, i.e.~morphisms
inducing an isomorphism on $d$-cohomology. 
\end{prop}

\begin{proof}
For Hochschild homology and the classical versions $HC^{[[u]]}$, $HC^{[[u,u^{-1}]]}$
and $HC^{[u^{-1}]}$ this is an immediate consequence of
Proposition~\ref{prop:quism-invariance}. By Corollary~\eqref{cor:cyccochain},
this covers all the $8$ verstions of cyclic homology. 
\end{proof}

One important application of this proposition concerns the analytic and
algebraic cyclic cochain complexes in Examples~\ref{ex:forms}
and~\ref{ex:algebras} built from the de Rham algebra $\Om^*(X)$ of
a manifold $X$. 
Recall from Lemma~\ref{lem:cross-product} that the exterior cross
products $\phi_n:\Om^*(X)^{\otimes(n+1)}\to\Om^*(X^{n+1})$ define a
morphism of cyclic cochain complexes. 
%and thus a morphism 
%$$
%   \phi:\Bigl(\bigoplus_{n\geq 0}\Om^*(X)^{\otimes n+1},d_H,B\Bigr)
%   \to \Bigl(\bigoplus_{n\geq 0}\Om^*(X^{n+1}),d_H,B\Bigr).
%$$
Since $\phi_n$ induces an isomorphism on $d$-homology for all $n$ by the
K\"unneth formula, Proposition~\ref{prop:quism-cyclicchains} implies

\begin{corollary}[Algebraic vs.~analytic cyclic homology]\label{cor:alg-ana}
For a manifold $X$, the map of mixed complexes
$$
   \phi:\Bigl(\bigoplus_{n\geq 0}\Om^*(X)^{\otimes(n+1)},d+b,B\Bigr)
   \to \Bigl(\bigoplus_{n\geq 0}\Om^*(X^{n+1}),d+b,B\Bigr)
$$
induces isomorphisms on Hochschild homology and on all versions of
cyclic homology. \hfill$\square$
\end{corollary}

%%%
\subsection{From differential graded algebras to mixed complexes}\label{sec:dga}
%%%

Recall from Example~\ref{ex:algebras} that each dga $(A,d)$
canonically gives rise to a cyclic cochain complex 
with $C_n(A)=A^{\otimes(n+1)}$.
% and the other operations as defined in Section~\ref{sec:cyc}. 
By Lemma~\ref{lem:cyccomplex-dgmod} it gives
rise to a mixed complex $(C(A)=\bigoplus_{n\geq 0}C_n(A),d+b,B)$
and a double complex $(C[[\theta,\theta^{-1}]],\delta_\theta)$. 
The homology $HH(A):=H(C(A),d+b)$ is the {\em Hochschild homology}
of $(A,d)$ and the homologies $HC_{[[u,u^{-1}]]}(A)$ etc are the various
versions of {\em cyclic homology} of $(A,d)$. In this subsection we
will study these cyclic homologies in more detail. 

Note that a morphism of dgas canonically induces a morphism of cyclic
cochain complexes. A quasi-isomorphism of dgas induces a
quasi-isomorphism of cyclic cochain complexes (by the K\"unneth formula
and exactness of the tensor functor on vector spaces), so by
Proposition~\ref{prop:quism-cyclicchains} it induces isomorphisms on 
Hochschild homology and all versions of cyclic homology. 
%If the degrees in $(A,d)$ are bounded from above (as is the case e.g.~for the
%de Rham complex of a finite dimensional manifold), then $C(A)$
%satisfies the hypothesis of Lemma~\ref{lem:our-iso}, and thus
%\begin{equation}\label{eq:our-iso-dga}
%   HC^*(A)[u^{-1}]\stackrel{\cong}\longrightarrow HC^*(A)[[u^{-1}]]. 
%\end{equation}
Moreover, by Corollary~\ref{cor:cyccochain} all versions of cyclic
homology of $(A,d)$ are either zero or isomorphic to one of the $3$ classical versions.

\begin{remark}
The invariance of all versions of cyclic homology under
quasi-isomorphisms of dgas implies, in particular, that the cyclic
homology of a {\em commutative} dga with $H^0(A)=\R$ can be computed
using its Sullivan minimal model. 
\end{remark}

%The following lemma will simplify our computations below. Its proof is
%a straightforward adaptation of the one in~\cite{Loday} (in fact, it
%holds more generally for cyclic chain complexes). 
%For a dga $(A,d)$ we define $\bar A:=A/(\R\cdot 1)$ and the {\em
%normalized mixed complex} $(\bar C(A)=\bigoplus\bar C_n(A),d+b,B)$ as the
%quotient $\bar C_n(A):=A\otimes\bar A^{\otimes n}$ with the induced
%operations $d+b$ and $B=sN$. 
%
%\begin{lemma}[normalized cyclic homology]
%For a dga $(A,d)$ the Hochschild homology $HH(A)$ and the versions
%$HC(A)[u]$, $HC(A)[u,u^{-1}]$ and $HC(A)[u^{-1}]$ of cyclic homology 
%agree with the corresponding homologies of the normalized mixed complex. \hfill$\square$
%\end{lemma}

{\bf Normalized and reduced cyclic homology of a dga.}
Consider a dga $A$ with its associated cyclic cochain complex
$(C_n(A)=A^{\otimes(n+1)},d_i,s_j,t_n,d)$ as in Example~\ref{ex:algebras}. 
As in~\cite{Loday} we denote by $D_n(A)\subset C_n(A)$ the linear subspace
spanned by words $(a_0|a_1|\cdots|a_n)$ with $a_i=1$ for some $i\geq 1$ 
(such words are called {\em degenerate}). So we get a short exact sequence
$$
   0 \to D_n(A)\to C_n(A)\to \ol{C}_n(A)\to 0
$$
where the quotient is given by
$$
   \ol{C}_n(A) = A\otimes\bar{A}^{\otimes n},\qquad \bar{A}=A/(\R\cdot 1).
$$
Note that the subspaces $D_n(A)\subset C_n(A)$ define a simplicial subcomplex
but not a cyclic one, i.e., it is preserved under the operations
$d_i,s_j,d$ but not under $t_n$. Nonetheless, one easily checks that the direct sums
$D(A)=\bigoplus_{n\geq 0}D_n(A)$ and $\ol{C}(A)=\bigoplus_{n\geq
  0}\ol{C}_n(A)$ give rise to a short exact sequence of mixed complexes
$$
   0 \longrightarrow \bigl(D(A),d+b,B\bigr)\longrightarrow
   \bigl(C(A),d+b,B\bigr)\stackrel{p}\longrightarrow
   \bigl(\ol{C}(A),d+b,\ol B\bigr)\longrightarrow 0
$$
with the induced map $\ol B=sN$. The quotient $\ol{C}(A)$ is called
the {\em normalized mixed complex}. It is a standard fact (see
e.g.~\cite[Proposition 1.15]{Loday}, which holds more generally for
cyclic chain complexes) that $(D(A),d+b)$ has vanishing 
homology, so $p:C(A)\to\ol{C}(A)$ is a quasi-isomorphism of mixed complexes.
We denote the resulting versions of {\em normalized cyclic homology} by
$H\ol{C}^{[u,u^{-1}]}(A)$ etc. 

The normalized mixed complex $\ol{C}(A)$ of a dga $A$ has the mixed
subcomplex $\ol{C}(\R)$ given by $\ol{C}_0(\R)=\R$ and
$\ol{C}_n(\R)=0$ for $n\geq 1$ with trivial
operations. Following~\cite{Loday}, we define the
{\em reduced mixed complex} $\ol{C}^\red(A)$ by the short exact sequence
$$
   0 \longrightarrow \ol{C}(\R) \longrightarrow \ol{C}(A) \longrightarrow \ol{C}^\red(A)\longrightarrow 0. 
$$
We denote the resulting versions of {\em reduced cyclic homology} by
$\ol{HC}^{[u,u^{-1}]}(A)$ etc. There is also a {\em reduced Connes version}
$\ol{HC}^\lambda(A)$ defined as the $(d+b)$-homology of the quotient
$\ol{C}^\lambda(A)$ by all words containing $1$ in some entry. 

A dga $A$ is called {\em augmented} if
$$
   A=\R\cdot 1\oplus\ol A
$$
for a dg ideal $\ol A\subset A$. For example, the de Rham complex
$\Om^*(X)$ of a connected manifold $X$ is augmented with
$\ol{A}=\Om^*(X,x_0)$ the forms restricting to zero on a basepoint $x_0$.
If $A$ is augmented the above exact sequence splits as
a direct sum of mixed complexes
\begin{equation}\label{eq:norm-red-complex}
   \ol{C}(A) = \ol{C}(\R)\oplus \ol{C}^\red(A),
\end{equation}
where $\ol{C}_0^\red(A)=\ol A$ and $\ol{C}_n^\red(A) = A\otimes\ol A^{\otimes n}$ for $n\geq 1$.
It follows that each version $\{u,u^{-1}\}$ of normalized cyclic homology
splits as
\begin{equation}\label{eq:norm-red-hom}
   H\ol{C}^{\{u,u^{-1}\}}(A) = H\ol{C}^{\{u,u^{-1}\}}(\R) \oplus \ol{HC}^{\{u,u^{-1}\}}(A).
\end{equation}
Note that for $A$ augmented the reduced Connes version is given by
$\ol{HC}^\lambda(A)=HC^\lambda(\ol A)$, viewing $\ol A$ as a
non-unital dga. 

\begin{prop}[normalized cyclic homology]\label{prop:normalized-cyc-hom}
For a dga $(A,d)$, the projection $p:C(A)\to\ol{C}(A)$ induces
isomorphisms from Hochschild homology $HH(A)$ and the classical versions
$HC^{[[u]]}(A)$, $HC^{[[u,u^{-1}]}(A)$ and $HC^{[u^{-1}]}(A)$ of cyclic homology 
to the corresponding homologies of the normalized mixed complex. 
%\begin{gather*}
%   HC^*(A)[[u]]\cong H\ol{C}^*(A)[[u]],\qquad 
%   HC^*(A)[u^{-1}]\cong H\ol{C}^*(A)[u^{-1}], \cr
%   HC^*(A)[[u,u^{-1}]\cong H\ol{C}^*(A)[[u,u^{-1}].
%\end{gather*}
The same holds for the versions $HC^{[[u^{-1}]]}(A)$ and $HC^{[[u,u^{-1}]]}(A)$ 
if $A$ is augmented. 
%it also induces isomorphisms 
%$$
%   HC^*(A)[[u^{-1}]]\cong H\ol{C}^*(A)[[u^{-1}]],\qquad 
%   HC^*(A)[[u,u^{-1}]]\cong H\ol{C}^*(A)[[u,u^{-1}]]. 
%$$
The remaining $3$ versions of cyclic homology are in general not
isomorphic to their normalized cylic homology. 
\end{prop}

\begin{proof}
The first assertion follows directly from Proposition~\ref{prop:quism-invariance} 
because $p$ is a quasi-isomorphism, and the last assertion follows
from Example~\ref{ex:point-deRham}. For the second assertion, suppose
that $A=\R\cdot 1\oplus \ol A$ is augmented and denote by
$p:A\to \ol A$ the projection.
%, so the map induced by $B$ on $\ol{C}(A)$ becomes $\ol B=psN$. 
Viewing $\ol A$ as a non-unital dga, it has an associated Hochschild
complex $(C(\ol A),d+b)$
and double complexes $(C(\ol A)\{\theta,\theta^{-1}\},\delta_\theta)$,
where $\{\theta,\theta^{-1}\}$ stands for any of the eight versions
$[\theta]$, $[\theta,\theta^{-1}]]$ etc. 
Note that the reduced cyclic complex has a canonical splitting as a
vector space
$$
   \ol{C}^\red(A) = C(\ol A)\oplus sC(\ol A),
$$
where $sC(\ol A)$ is generated by words with $1$ in the $0$-th slot
and elements from $\ol A$ in all others. Now the main observation is that the map
$$
   r((x+sy)u^n) := x\theta^{2n}+y\theta^{2n-1}
$$
defines a chain isomorphism
$$
   r:\Bigl(\bigl(C(\ol A)\oplus sC(\ol A)\bigr)\{u,u^{-1}\},d+b+\ol{B}u\Bigr)
   \stackrel{\cong}\longrightarrow \Bigl(C(\ol A)\{\theta,\theta^{-1}\},\delta_\theta\Bigr).
$$
Indeed, the map is clearly bijective, and the chain map property
follows from
\begin{align*}
   r(d+b+\ol Bu)\bigl((x+sy)u^n\bigr) 
   &= r\Bigl((d+b)(x+sy)u^n+\ol B(x+sy)u^{n+1}\Bigr) \cr
   &= r\Bigl(\bigr(dx++dsy+bx-sb'y+(1-t)y\bigr)u^n + sNxu^{n+1}\Bigr) \cr
   &= \bigl((d+b)x+(1-t)y\bigr)\theta^{2n} - (d+b')y\theta^{2n-1} + Nx\theta^{2n+1}
\end{align*}
and
\begin{align*}
   \delta_\theta r\bigl((x+sy)u^n\bigr)
   &= \delta_\theta(x\theta^{2n}+y\theta^{2n-1}) \cr
   &= (d+b)x\theta^{2n} - (d+b')y\theta^{2n-1} + Nx\theta^{2n+1} + (1-t)y\theta^{2n}\cr
   &= \bigl((d+b)x+(1-t)y\bigr)\theta^{2n} - (d+b')y\theta^{2n-1} + Nx\theta^{2n+1}.
\end{align*}
Combining this with the splitting~\eqref{eq:norm-red-complex}, we obtain a chain isomorphism
$$
   \id\oplus r:\Bigl(\ol{C}(A)\{u,u^{-1}\},d+b+\ol{B}u\Bigr)
   \stackrel{\cong}\longrightarrow \ol{C}(\R)\{u,u^{-1}\} \oplus
   \Bigl(C(\ol A)\{\theta,\theta^{-1}\},\delta_\theta\Bigr). 
$$
The induced isomorphisms on homology for the $[u^{-1}]$ and
$[[u^{-1}]]$ versions fit into the commuting diagram
$$
\xymatrix
@C=40pt
{
   HC_*^{[u^{-1}]}(A) \ar[d]_{\phi_*}^\cong \ar[r]^{p_*}_\cong & H\ol{C}_*^{[u^{-1}]}(A) \ar[d]_{\phi_*}
   \ar[r]^{(\id\oplus r)_*\ \ \ \ \ \ \ \ }_{\cong\ \ \ \ \ \ \ \ \ \ } 
   & \R[u^{-1}]\oplus HC_*^{[\theta^{-1}]}(\ol A) \ar[d]_{\id\oplus\phi_*}^\cong \\
   HC_*^{[[u^{-1}]]}(A) \ar[r]^{p_*} & H\ol{C}_*^{[[u^{-1}]]}(A)
   \ar[r]^{(\id\oplus r)_*\ \ \ \ \ \ \ \ }_{\cong\ \ \ \ \ \ \ \ \ \ } 
   & \R[u^{-1}]\oplus HC_*^{[[\theta^{-1}]]}(\ol A)\,. 
}
$$
Here the vertical maps $\phi_*$ come from Lemma~\ref{lem:our-iso} and
the upper horizontal map $p_*$ is an isomorphism by the first assertion above. 
It follows that the lower horizontal map $p_*$ is an isomorphism as well.
An analogous argument for the $[[u,u^{-1}]$ and $[[u,u^{-1}]]$
versions concludes the proof. 
\end{proof}

\begin{cor}[reduced cyclic homology]\label{cor:reduced-cyc-hom}
For an augmented dga $(A,d)$ we have a canonical splitting
$$
   HC^{[[u]]}_*(A) = H\ol{C}^{[[u]]}_*(\R)\oplus \ol{HC}^{[[u]]}_*(A),
$$
and similarly for the $[[u,u^{-1}]$, $[u^{-1}]$, $[[u^{-1}]]$ and
$[[u,u^{-1}]]$ versions. 
Moreover, $\ol{HC}^{[u,u^{-1}]]}_*(A)=0$ and the connecting
homomorphism in the corresponding tautological exact sequence splits
into isomorphisms
$$
   \ol{HC}^{[[u^{-1}]]}_{*+1}(A) \cong \ol{HC}^\lambda_*(A)=HC^\lambda_*(\ol A)\cong \ol{HC}^{[u]}_{*}(A). 
$$
\end{cor}

\begin{proof}
The first assertion follows immediately from the
splitting~\eqref{eq:norm-red-hom} and Proposition~\ref{prop:normalized-cyc-hom}.
%The proof of~\cite[Proposition 2.2.14]{Loday}
The second assertion follows from the following variant of the
diagram in Lemma~\ref{lem:cyccomplex-computation} in reduced cyclic homology:
\begin{equation*}%\label{eq:cyclic-cochain}
\xymatrix{
   \ol{HC}_*^{[[u^{-1}]]}(A) \ar[d]_{\cong}^{r_*} \ar[r]^{\ol{B}_0} & \ol{HC}_{*-1}^{[u]}(A) \ar[d]_{\cong}^{r_*}\\
   HC_*^{[[\theta^{-1}]]}(\ol A) \ar[d]_{\cong}^{e_*}
   \ar[r]^{\ol{B}_0}_{\cong} & HC_{*-1}^{[\theta]}(\ol A) \\
   H_*\bigl(C(\ol A)/\im(1-t),d+b\bigr) \ar[r]^{N_*}_{\cong} & H_{*}(\im
   N,d+b') \ar[u]^{\cong}_{s_*}\,.
}
\end{equation*}
Here $\ol B=sN$, the maps $r$ are the chain isomorphisms from the proof of
Proposition~\ref{prop:normalized-cyc-hom}, and the other maps are
isomorphisms by Lemma~\ref{lem:cyccomplex-computation} applied to the
cyclic complex of $\ol A$. Since $H_{*+1}\bigl(C(\ol A)/\im(1-t),d+b\bigr)
= HC^\lambda_*(\ol A)=\ol{HC}^\lambda_*(A)$, this proves the chain of isomorphisms. 
Vanishing of $\ol{HC}^{[u,u^{-1}]]}_*(A)$ now follows from the
tautological exact sequence. 
\end{proof}

%%%
\subsection{From $S^1$-spaces to mixed complexes}\label{subsec:circle-to-mixed}
%%%

In this subsection we recall the mixed complexes arising from
$S^1$-spaces and their properties.

{\bf Topological $S^1$-spaces. }
Let $Y$ be a topological space with an $S^1$-action $\phi:S^1\times
Y\rightarrow Y$. We make the singular cochain complex $(C^*(Y),d)$
a mixed complex by introducing the second differential
$P:C^*(Y)\rightarrow C^{*-1}(Y)$ as follows. Let $\times:C_m(X)\otimes
C_n(Y)\to C_{m+n}(X\times Y)$ be the Eilenberg-MacLane shuffle product,
associating to $f:\Delta^m\to X$ and $g:\Delta^n\to Y$ the map
$f\times g:\Delta^m\times\Delta^n\to X\times Y$, $(p,q)\mapsto(f(p),g(q))$ with a
canonical subdivision of $\Delta^m\times\Delta^n$ into simplices
(using shuffles of the variables, see e.g.~\cite{Hatcher}). 
The shuffle product with the fundamental cycle $F_{S^1}:\Delta^1\to S^1$
defines a map $C_*(Y)\to C_{*+1}(S^1\times Y)$, $c\mapsto
F_{S^1}\times c$, whose composition with $\phi_*:C_*(S^1\times Y)\to
C_*(Y)$ gives a map
\begin{equation}\label{eq:Q}
   Q:C_*(Y)\to C_{*+1}(Y),\qquad c\mapsto \phi_*(F_{S^1}\times c). 
\end{equation}
The operator $P:C^*(Y)\rightarrow C^{*-1}(Y)$ is dual to $Q$, i.e.
$$
   \la P\alpha,c\ra := \la\alpha,\phi_*(F_{S^1}\times c)\ra,\qquad
   \alpha\in C^*(Y),\ c\in C_{*-1}(Y).
$$
%In words: $P$ is the dual operator to multiplication with $F_{S^1}$
%left composed with pushforward by the action. 

\begin{lemma}\label{lem:S1-space-dgmod}
$(C^*(Y),d,P)$ is a mixed complex whose homology $H^*(Y,d)$ is the
singular cohomology of $Y$.\footnote{In this paper, all singular
  (co)homology is with $\R$-coefficients.}  
\end{lemma}

\begin{proof}
To show $P^2=0$, we compute for $c\in C_*(Y)$ using associativity of
the shuffle product
$$
   Q^2c = \phi_*(F_{S^1}\times\phi_*(F_{S^1}\times c)) =
   \phi_*(\psi_*(F_{S^1}\times F_{S^1})\times c),
$$
where $\psi:S^1\times S^1\to S^1$ is the product $(\sigma,\tau)\mapsto
\sigma+\tau$. Now by definition of the shuffle product $F_{S^1}\times F_{S^1}$
is the difference of the singular simplices $f,g:\Delta^2\to S^1\times
S^1$ given by $f(s,t)=(s,t)$ and $g(s,t)=(t,s)$. This shows that
$\psi_*(F_{S^1}\times F_{S^1})=0\in C_2(S^1)$, thus $Q^2=0$ and
therefore $P^2=0$. For anticommutation of $d$ and $P$ note that 
$$
   \p(F_{S^1}\times c)= \p F_{S^1}\times c-F_{S^1}\times\p c =
   -F_{S^1}\times c
$$ 
because $\p F_{S^1}=0$. Applying $\phi_*$ and dualizing we find $\p
Q+Q\p=0$ and $dP+Pd=0$. The statement about cohomology is clear by
definition. 
\end{proof}

We denote the cyclic homologies of the mixed complex $(C^*(Y),d,P)$ by
$H^*_{[[u,u^{-1}]]}(Y)$ etc.\footnote{
Here we use cohomological notation because $H^*_{[[u,u^{-1}]]}(Y)$
%are variants of the singular {\em cohomology} of $Y$, see
%Lemma~\ref{lem:S1-space-dgmod} and Proposition~\ref{prop:Borel}.
is the cyclic {\em cohomology} of the mixed complex $(C_{-*}(Y),\p,Q)$
in the sense of Section~\ref{sec:duality}.
}

{\bf Relation to the Borel construction. }
According to the Borel construction, the {\em $S^1$-equivariant cohomology} of a topological
$S^1$-space $Y$ is defined as
$$
   H_{S^1}^*(Y) := H^*(Y\times_{S^1}ES^1),
$$
where $ES^1\to BS^1$ is the universal $S^1$-bundle and $Y\times_{S^1}ES^1 =
(Y\times ES^1)/S^1$ is the quotient by the diagonal circle action. 
In the following proposition the first isomorphism is shown in the
discussion following Lemma 5.1 in~\cite{Jones}, and the second
isomorphism holds because the complex $C^*(Y)$ lives in nonnegative degrees.

\begin{proposition}[Jones~\cite{Jones}]\label{prop:Borel}
For each topological $S^1$-space $Y$ we have canonical isomorphisms 
\begin{equation*}
\begin{CD}
   H_{S^1}^*(Y) @>{\cong}>> H^{*}_{[u]}(Y) @>{\cong}>> H^{*}_{[[u]]}(Y).
\end{CD}
\end{equation*}
\hfill$\square$
\end{proposition}

{\bf Projection to a point. }
For an $S^1$-space $Y$ with a fixed point $y_0$ consider the inclusion and projection
$$
   \pt\stackrel{\iota}\longrightarrow
   Y\stackrel{\pi}\longrightarrow\pt,\qquad \iota(\pt)=y_0.
$$
These maps are $S^1$-equivariant and satisfy $\pi\iota=\id$, so they
induce morphisms of mixed complexes
$$
   \Bigl(C^*(\pt)=\R,d=0,P=0\Bigr)\stackrel{\pi^*}\longrightarrow
   (C^*(Y),d,P)\stackrel{\iota^*}\longrightarrow
   \Bigl(C^*(\pt)=\R,d=0,P=0\Bigr)
$$
satisfying $\iota^*\pi^*=\id$. These maps are functorial and the
induced maps on cyclic homology are compatible with the Gysin and
tautological sequences in the obvious way. For example, one tautological
sequence yields the following commuting diagram, where all the vertical maps
$\pi^*$ are injective and we have surjective vertical maps $\iota^*$
in the other direction:
\begin{equation}\label{eq:proj-point}
\begin{CD}
   H^*_{[[u,u^{-1}]}(Y) @>{p_*}>> H^*_{[u^{-1}]}(Y) @>{P_0}>>
       H^{*-1}_{[[u]]}(Y) @>{\cdot u}>> H^{*+1}_{[[u,u^{-1}]}(Y) \\
  @AA{\pi^*}A @AA{\pi^*}A @AA{\pi^*}A @AA{\pi^*}A \\
  \R[u,u^{-1}] @>{p_*}>> \R[u^{-1}] @>0>>
       \R[u] @>{\cdot u}>> \R[u,u^{-1}]\;. 
\end{CD}
\end{equation}
For a loop space $Y=LX$, the first group in this diagram
can be explicitly computed:

\begin{thm}[Goodwillie~\cite{Goodwillie}]\label{thm:Goodwillie}
For a path connected space $X$, the group $H^*_{[[u,u^{-1}]}(LX)$
depends only on $\pi_1(X)$. In particular, for $X$ simply connected 
the projection $LX\to\pt$ to a point induces an isomorphism
$$
   H^*_{[[u,u^{-1}]}(LX) \cong H^*_{[[u,u^{-1}]}(\pt) = \R[u,u^{-1}].    
$$
\end{thm}

{\bf Smooth $S^1$-spaces. }
Let now $Y$ be a differentiable space in the sense of Chen~\cite{Chen77}
equipped with a smooth circle action. For example, $Y$ could be a
finite dimensional manifold or an infinite dimensional Banach or
Fr\'echet manifold. Our main example is the space $Y=LX$ of smooth
maps $S^1\to X$ into a finite dimensional manifold $X$ with the
natural circle action $\phi(s,\gamma):=\gamma(\cdot+s)$ by
reparametrization. Following Chen~\cite{Chen77}, to such a {\em smooth
$S^1$-space} $Y$ we associate a mixed complex $(\Om^*(Y),d,P)$ as
follows. Let $(\Om^*(Y),d)$ be the space of differential forms on $Y$
with exterior derivative. The inner product with the vector field $v$ on
$Y$ generating the circle action gives a map
$\iota:\Om^*(Y)\to\Om^{*-1}(Y)$. By Cartan's formula, $d\iota+\iota
d=L_v$ is the Lie derivative in direction of $v$. Let
$A:\Om^*(Y)\to\Om^*(Y)$ be the operator averaging over the circle
action. In view of the relations $\iota A=A\iota$ and $dA=Ad$,
Cartan's formula implies that 
$$
   P:=A\iota:\Om^*(Y)\to\Om^{*-1}(Y)
$$
satisfies $dP+Pd=0$. The relation $\iota^2=0$ implies $P^2=0$, so 
$(\Om^*(Y),d,P)$ is indeed a mixed complex. 
It is shown in~\cite{Chen77} that for ``nice'' smooth $S^1$-spaces
(such as loop spaces of manifolds) the homology of $(\Om^*(Y),d)$ is
isomorphic to the singular cohomology $H^*(Y;\R)$. 

To a smooth $S^1$-space $Y$ we can associate another canonical mixed
complex $(\Om_{S^1}^*(Y),d,\iota)$, where $\Om^*_{S^1}(Y)$ denotes the
space of $S^1$-invariant forms.  

\begin{proposition}[Jones--Petrack~\cite{Jones-Petrack}]
The inclusion $\Om_{S^1}^*(Y)\into \Om^*(Y)$ induces a homotopy
equivalence of mixed complexes 
$$
   (\Om_{S^1}^*(Y),d,\iota) \to (\Om^*(Y),d,P).
$$
In particular, it induces isomorphisms on all versions of cyclic
homology. 
\end{proposition}

{\bf Fixed point free $S^1$-actions. }
Let $Y$ be a smooth $S^1$-space without fixed points. Then there 
exists a connection form, i.e., an invariant $1$-form $\alpha\in\Om^1_{S^1}(Y)$
satisfying $\alpha(v)=1$ for the vector field $v$ generating the action. 
The wedge product with $\alpha$ then defines a chain homotopy
$H:\Om^*_{S^1}(Y)\to \Om^{*+1}_{S^1}(Y)$
satisfying
$$
   \iota H+H\iota = \id.
$$
Consider the tautological sequence
\begin{equation*}
\begin{CD}
   H^{*+1}_{[u,u^{-1}]]}(Y) @>{p_*}>> H^{*+1}_{[[u^{-1}]]}(Y) @>{P_0}>>
   H^{*}_{[u]}(Y) @>{\cdot u}>> H^{*+2}_{[u,u^{-1}]]}(Y).
\end{CD}
\end{equation*}
The map $P_0$ is the composition of the first two maps in 
\begin{equation*}
\begin{CD}
   H^{*+1}_{[[u^{-1}]]}(Y) @>>> H^*(\ker\iota)=H^*(Y/S^1) 
   @>>> H^{*}_{[u]}(Y) @>{\cong}>> H^{*}_{[[u]]}(Y),
\end{CD}
\end{equation*}
where the last map is an isomorphism because the complex $\Om_{S^1}^*(Y)$
lives in nonnegative degrees. The homotopy formula for $H$ shows that
the double complex $(\Om_{S^1}^*(Y)[u,u^{-1}]],d+\iota u)$ for the computation
of $H^{*}_{[u,u^{-1}]]}(Y)$ has exact rows. Now a standard zigzag
argument as in the proof of Lemma~\ref{lem:cyccomplex-computation}
yields 

\begin{lemma}\label{lem:free-action}
Let $Y$ be a smooth $S^1$-space without fixed points. Then
$H^{*}_{[u,u^{-1}]]}(Y)=0$ and we have canonical isomorphisms 
\begin{equation*}
\begin{CD}
   H^{*+1}_{[[u^{-1}]]}(Y) @>{\cong}>> H^*(Y/S^1) 
   @>{\cong}>> H^{*}_{[u]}(Y) @>{\cong}>> H^{*}_{[[u]]}(Y).
\end{CD}
\end{equation*}
\hfill$\square$
\end{lemma}

\begin{remark}
Consider a smooth $S^1$-space $Y$ (possibly with fixed points). 
Then the canonical projection $Y\times ES^1\to Y$ induces a morphism of
mixed complexes 
$$
   (\Om^*_{S^1}(Y),d,\iota)\to (\Om^*_{S^1}(Y\times ES^1),d,\iota)
$$
which induces an isomorphism on $d$-homology, and thus on the
$[[u]]$, $[u^{-1}]$ and $[[u,u^{-1}]$ versions of cyclic homology. 
Since $Y\times ES^1$ has no fixed points, applying
Lemma~\ref{lem:free-action} to it provides an alternative proof of
Proposition~\ref{prop:Borel} in the smooth case. 
\end{remark}

{\bf Fixed point localization. }
The version $H^*_{[[u,u^{-1}]]}(Y)$ satisfies the following fixed
point localizaton theorem. 

\begin{thm}[Jones--Petrack~\cite{Jones-Petrack}]\label{thm:JP}
Let $Y$ be a smooth $S^1$-space whose fixed point set $F\subset Y$ is
a smooth submanifold which has an $S^1$-invariant tubular neighbourhood.
Then the inclusion $c:F\into Y$ induces an isomorphism
$$
   c^*:H^*_{[[u,u^{-1}]]}(Y)\stackrel{\cong}\longrightarrow H^*(F)\otimes\R[u,u^{-1}].
$$
In particular, for a manifold $X$ the inclusion $c:X\into LX$ of the
constant loops induces an isomorphism
$$
   c^*:H^*_{[[u,u^{-1}]]}(LX)\stackrel{\cong}\longrightarrow H^*(X)\otimes\R[u,u^{-1}].
$$
\end{thm}

%%%
\subsection{Examples}
%%%

In this subsection we work out the different flavours of cyclic
homology for some examples. In the case of a dga $A$ we will use the
equivalent sign convention~\eqref{eq:cyc-signs-Loday}, dropping the
$\sim$'s, which spells out as
\begin{align*}
%   d(a_0|\dots|a_n) &= \sum_{i=0}^n(-1)^{n+|a_0|+\dots+|a_{i-1}|}
%   (a_0|\dots|da_i|\dots|a_n), \cr 
   b'(a_0|\dots|a_n) &= \sum_{i=0}^{n-1}(-1)^i (a_0|\dots|
   a_ia_{i+1}|\dots|a_n), \cr
   b(a_0|\dots|a_n) &= b'(a_0|\dots| a_n) + (-1)^{n+\deg a_n(\deg
     a_0+\dots+\deg a_{n-1})} (a_na_0|a_1|\dots|a_{n-1}), \cr 
   t(a_0|\dots|a_n) &= (-1)^{n+\deg a_n(\deg a_0+\dots+\deg a_{n-1})}
   (a_n|a_0|\dots|a_{n-1}), \cr 
%   N(a_0|\cdots|a_n) &=
%   \sum_{i=0}^n(-1)^{ni+(|a_0|+\cdots+|a_{i-1}|)(|a_i|+\cdots+|a_n|)} 
%   (a_i|\cdots|a_n|a_0|\cdots|a_{i-1}), \cr
   s(a_0|\dots|a_n) &= (1|a_0|\dots|a_n),\cr
%   B(a_0|\cdots|a_n) &=
%   \sum_{i=0}^n(-1)^{ni+(|a_0|+\cdots+|a_{i-1}|)(|a_i|+\cdots+|a_n|)} \cr 
%   & \ \ \ \Bigl[(1|a_i|\cdots|a_n|a_0|\cdots|a_{i-1})
%   - (a_i|1|a_{i+1}|\cdots|a_n|a_0|\cdots|a_{i-1})\Bigr].
\end{align*}
%Here we choose $t$ as in~\cite{Loday}, which differs from~\cite{Jones}.

\begin{example}[de Rham complex of a point]\label{ex:point-deRham}
Consider the trivial dga $(A=\R,d=0)$ sitting in degree $0$, which is the de Rham complex of
a point. Its Hochschild complex $C(\R)$ has the basis
$1^n:=(1|\dots|1)$ of words with $n$ $1$'s of degree $|1^n|=1-n$, for
$n\geq 1$. From the definitions (and being careful about signs) we
read off the operations
\begin{align*}
   b(1^n) &= \begin{cases}
      1^{n-1} & n\geq 3\text{ odd},\\
      0 & \text{else}, 
   \end{cases} \cr
   t(1^n) &= (-1)^{n-1}1^n, \cr
   N(1^n) &= \begin{cases}
      n1^{n} & n\text{ odd},\\
      0 & \text{else},
   \end{cases} \cr
   sN(1^n) &= \begin{cases}
      n1^{n+1} & n\text{ odd},\\
      0 & \text{else},
   \end{cases} \cr
   B(1^n) &= \begin{cases}
      2n1^{n+1} & n\text{ odd},\\
      0 & \text{else}.
   \end{cases}
\end{align*}
This gives us the mixed complex $(C=C(\R),b,B)$. As shown in Figure~\ref{fig2},
cycles in the complex $(C[[u,u^{-1}]],\delta_u=b+uB)$ consist of zigzags in the
lower half plane starting on the horizontal axis and extending indefinitely to the right. 

\begin{figure}
\labellist
\small
\pinlabel* $1$ at 191 316 
\pinlabel* $1^2$ at 189 255
\pinlabel* $1^3$ at 189 203 
\pinlabel* $1^4$ at 189 151 
\pinlabel* $1^5$ at 189 98
\pinlabel* $u^21$ at 305 317
\pinlabel* $u^21^2$ at 318 262
\pinlabel* $u^21^3$ at 318 208
\pinlabel* $u^21^4$ at 318 158
\pinlabel* $u^21^5$ at 318 104
\pinlabel* $u1$ at 252 316
\pinlabel* $u1^2$ at 263 262
\pinlabel* $u1^3$ at 263 208
\pinlabel* $u1^4$ at 263 158
\pinlabel* $u1^5$ at 263 104
\pinlabel* $b$ at 296 120
\pinlabel* $b$ at 244 226
\pinlabel* $b$ at 205 226
\pinlabel* $b$ at 205 120
\pinlabel* $B$ at 230 183
\pinlabel* $B$ at 230 287
\pinlabel* $B$ at 282 183
\pinlabel* $B$ at 331 80
\endlabellist 
\centering
  \includegraphics[width=1.0\textwidth]{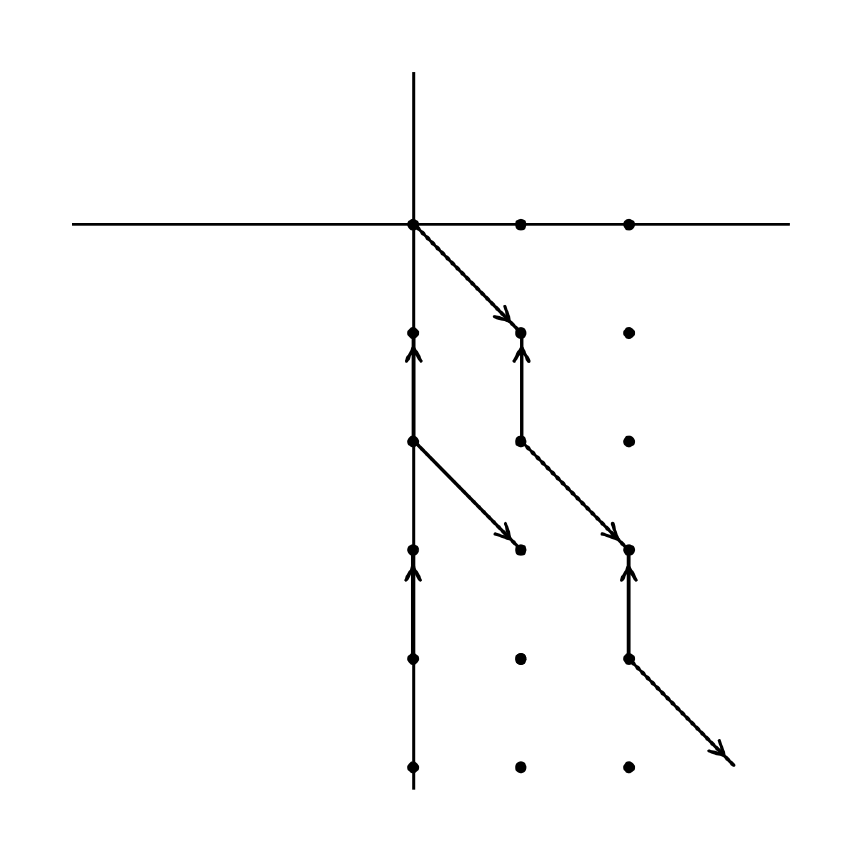}
  \caption{Cyclic complex of the trivial dga $(A=\R,d=0)$}
  \label{fig2}
\end{figure}
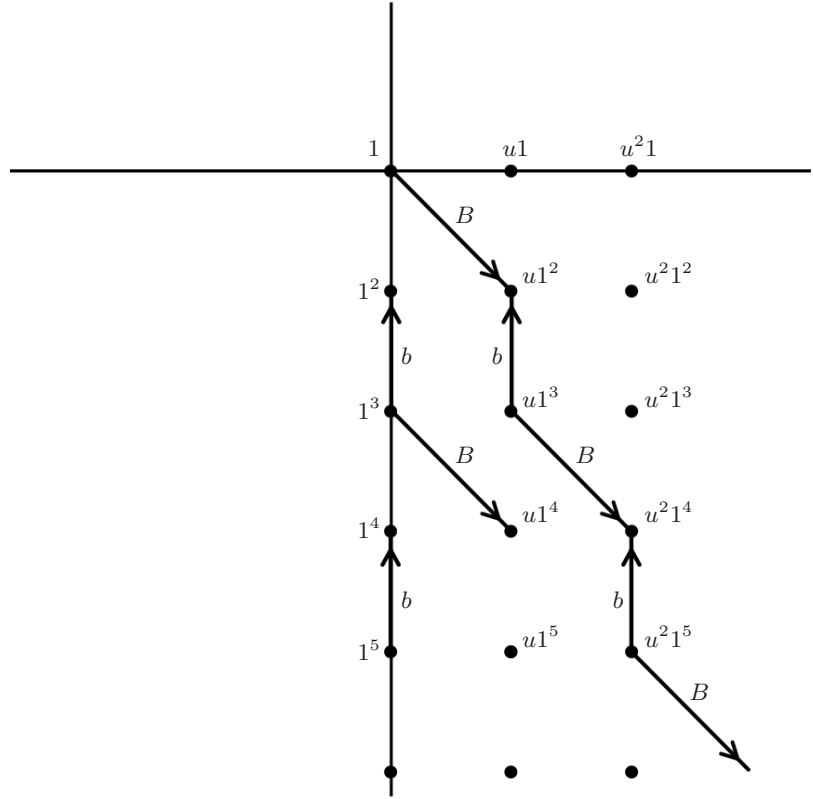

Thus it makes no difference whether we consider polynomials or
power series in $u^{-1}$. Using this, we read off the Hochschild and
cyclic homology groups (where $[i]$ denotes degree shift by $i$)
\begin{align*}
   &H(C,\delta)=\R, \cr
   &HC^{[u]}=\R[u^{-1}][-1]\quad\text{generated by }1^2\sim u1^4\sim\cdots, \cr
   &HC^{[u,u^{-1}]}=HC^{[u,u^{-1}]]}=0, \cr
   &HC^{[u^{-1}]}=HC^{[[u^{-1}]]}=\R[u^{-1}]\quad\text{generated by }u^{-k}1+u^{-k+1}1^3+\cdots+1^{2k+1}, k\geq 0, \cr
   &HC^{[[u]]}=\R[u]\quad\text{generated by }1+u1^3+u^21^5\cdots, \cr
   &HC^{[[u,u^{-1}]]}=HC^{[[u,u^{-1}]}=\R[u,u^{-1}]\quad\text{generated by
  }1+u1^3+u^21^5\cdots. 
\end{align*}
So we see two types of tautological sequences:
\begin{gather}\label{eq:taut-ex1}
\xymatrix{
    0 \ar[r] & HC_{*-2}^{[[u]]} \ar@{=}[d] \ar[r]^{\cdot u} & HC_*^{[[u,u^{-1}]}
      \ar@{=}[d] \ar[r]^{p_*} & HC_*^{[u^{-1}]} \ar@{=}[d] \ar[r] & 0 \\
    0 \ar[r] & \R[u][-2] \ar[r]^{\cdot u} & \R[u,u^{-1}] \ar[r]^{p_*}
    & \R[u^{-1}] \ar[r] & 0
}
\end{gather}
and
\begin{gather*}
\xymatrix{
   HC_*^{[u,u^{-1}]} \ar@{=}[d] \ar[r] & HC_*^{[u^{-1}]} \ar@{=}[d]
   \ar[r]^{D_0}_{\cong} & HC_{*-1}^{[u]} \ar@{=}[d] \ar[r] &
   HC_{*+1}^{[u,u^{-1}]} \ar@{=}[d] \\
   0 \ar[r] & \R[u^{-1}]\cdot 1 \ar[r]^{1\mapsto 1^2}_{\cong} & \R[u^{-1}]\cdot
   1^2 \ar[r] & 0\;.
}
\end{gather*}
The normalized complex is given by $\ol{C}_0(\R)=\R$ sitting in degree
$0$ and $\ol{C}_n(\R)=0$ for $n\geq 1$, so the normalized cylic homologies are
\begin{gather*}
   H\ol{C}^{[u]} = H\ol{C}^{[[u]]} = \R[u],\qquad 
   H\ol{C}^{[u^{-1}]} = H\ol{C}^{[[u^{-1}]]} = \R[u^{-1}],\cr
   H\ol{C}^{[u,u^{-1}]} = H\ol{C}^{[[u,u^{-1}]} =
   H\ol{C}^{[u,u^{-1}]]} = H\ol{C}^{[[u,u^{-1}]]} = \R[u,u^{-1}].  
\end{gather*}
\end{example}

\begin{example}[singular cochains on a point]\label{ex:point-sing}
Consider the trivial mixed complex $(C=\R,\delta=D=0)$, which is the singular
cochain complex of a point viewed as an $S^1$-space. The homologies
can be directly read off to be
\begin{align*}
   &H(C,\delta)=\R, \cr
   &HC^{[u]}=HC^{[[u]]}=\R[u], \cr
   &HC^{[u^{-1}]}=HC^{[[u^{-1}]]}=\R[u^{-1}], \cr
   &HC^{[u,u^{-1}]}=HC^{[u,u^{-1}]]}=HC^{[[u,u^{-1}]}=HC^{[[u,u^{-1}]]}=\R[u,u^{-1}].
\end{align*}
Here all tautological sequences look like~\eqref{eq:taut-ex1} above. 
\end{example}

\begin{example}[singular cochains on $ES^1$]\label{ex:ES1-sing}
Consider the mixed complex $(C,\delta,D)$, where $C=\Lambda[\alpha,\beta]$ is
the exterior algebra in two generators of degrees $|\alpha|=1$ and
$|\beta|=2$ with the operations defined by the Leibniz rule and
$$
   \delta\alpha=\beta,\quad
   \delta\beta=0,\quad
   D\alpha=1,\quad
   D\beta=0.
$$
This is the Cartan--Weil model for the classifying space
$ES^1=S^\infty$ with its standard $S^1$-action,
see~\cite{Atiyah-Bott84}. As shown in Figure~\ref{fig3}, 
cycles in the complex $(C[[u,u^{-1}]],\delta_u=\delta+uD)$ consist of zigzags
in the upper half plane starting on the horizontal axis and extending
indefinitely to the left. 

\begin{figure}
\labellist
\small 
\pinlabel* $1$ at 190 65
\pinlabel* $\delta$ at 203 147
\pinlabel* $\delta$ at 203 254
\pinlabel* $D$ at 225 108
\pinlabel* $u^{-1}1$ at 141 65
\pinlabel* $u^{-1}\alpha$ at 125 124
\pinlabel* $u^{-1}\beta$ at 125 174
\pinlabel* $u^{-2}\alpha\beta$ at 72 225
\pinlabel* $\alpha$ at 184 126
\pinlabel* $\beta$ at 184 177
\pinlabel* $\alpha\beta$ at 182 231
\pinlabel* ${\beta}^2$ at 184 283
\pinlabel* $u1$ at 248 65
\endlabellist 
\centering
  \includegraphics[width=1.0\textwidth]{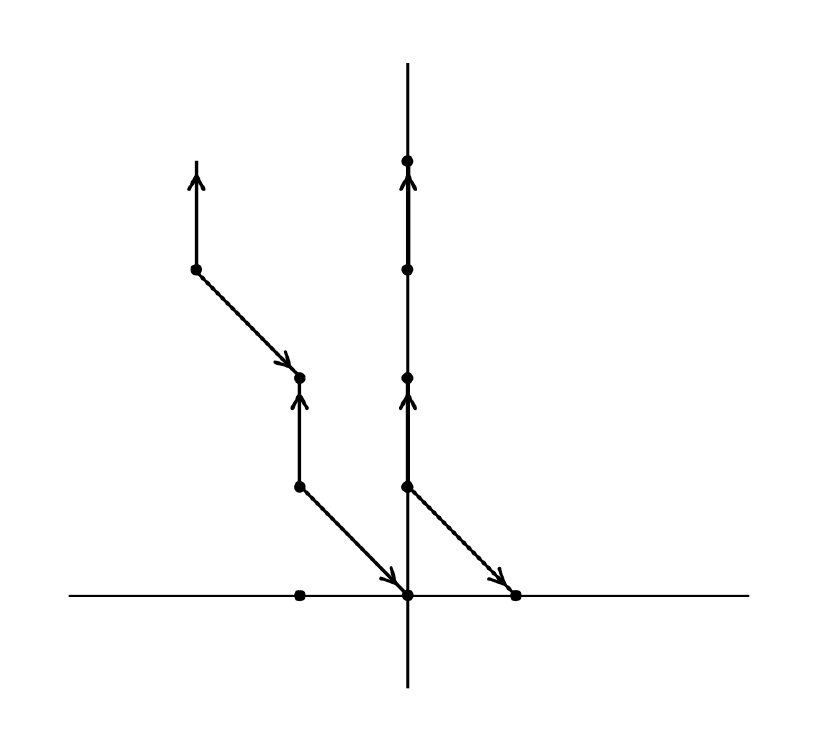}
  \caption{Cyclic complex of singular cochains on $ES^1$}
  \label{fig3}
\end{figure}
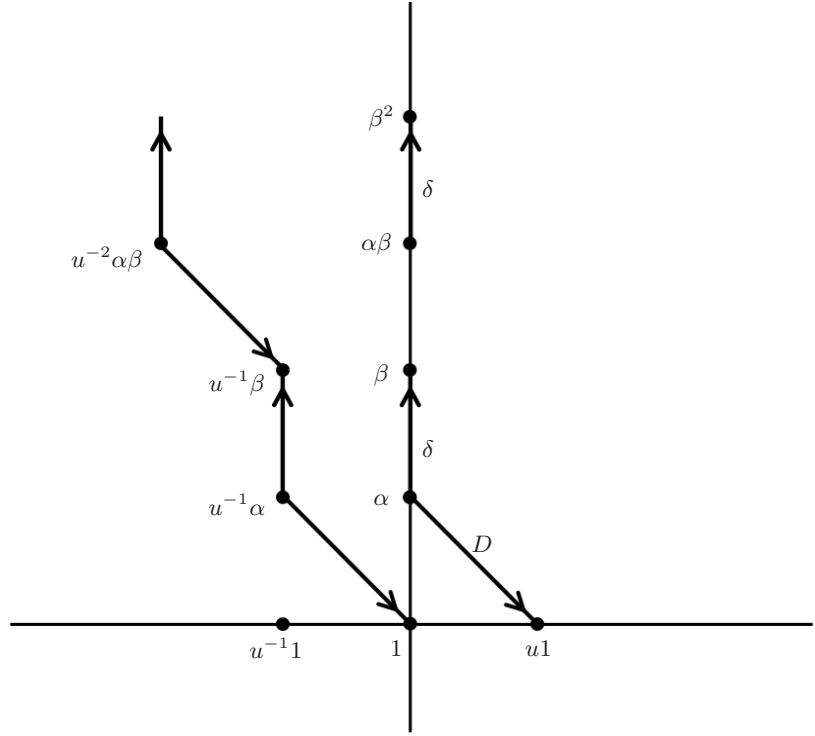

Thus it makes no difference whether we consider polynomials or
power series in $u$. Using this, we read off the Hochschild and
cyclic homology groups
\begin{align*}
   &H(C,\delta)=\R, \cr
   &HC^{[[u]]}=HC^{[u]}=\R[u]\quad\text{generated by $1$ (note that $\beta\sim u$)}, \cr
   &HC^{[[u,u^{-1}]]}=HC^{[u,u^{-1}]]}=0, \cr
   &HC^{[[u^{-1}]]}=\R[u][1]\quad\text{generated by }\alpha(1+u^{-1}\beta+u^{-2}\beta^2+\cdots, \cr
   &HC^{[[u,u^{-1}]}=HC^{[u,u^{-1}]}=\R[u,u^{-1}]\quad\text{generated by }1, \cr
   &HC^{[u^{-1}]}=\R[u^{-1}]\quad\text{generated by }1,u^{-1},u^{-2},\dots. 
\end{align*}
So we see two types of tautological sequences:
\begin{gather*}
\begin{CD}
    @>0>> HC_{*-2}^{[u]} @>{\cdot u}>> HC_*^{[u,u^{-1}]} @>{p_*}>> HC_*^{[u^{-1}]} @>0>> \\
    && @V=VV @V=VV @V=VV \\
    @>0>> \R[u][-2] @>{\cdot u}>> \R[u,u^{-1}] @>{p_*}>> \R[u^{-1}] @>0>>
\end{CD}
\end{gather*}
and
\begin{gather*}
\begin{CD}
   HC_*^{[u,u^{-1}]]} @>>> HC_*^{[[u^{-1}]]} @>{D_0}>{\cong}> HC_{*-1}^{[u]} @>>> HC_{*+1}^{[u,u^{-1}]]}\\
   @V=VV @V=VV @V=VV @V=VV \\
   0 @>>> \R[u^{-1}]\cdot \alpha @>{\alpha\mapsto 1}>{\cong}> \R[u^{-1}]\cdot
   1 @>>> 0\;. \\
\end{CD}
\end{gather*}
\end{example}

\begin{lemma}\label{lem:examples-quiso}
The mixed complexes in Examples~\eqref{ex:point-deRham},
\eqref{ex:point-sing} and~\eqref{ex:ES1-sing} are quasi-isomorphic.
\end{lemma}

\begin{proof}
The morphisms of mixed complexes
$$
   (C(\R),\delta=b,D=B)\to (\R,\delta=D=0),\qquad 1^n\mapsto\begin{cases}
      1 & n=1,\\ 0 & n\geq 2 \end{cases}
$$
from Example~\eqref{ex:point-deRham} to~\eqref{ex:point-sing} and
$$
   (\R,\delta=D=0)\to (\Lambda[\theta,\om],\delta,D),\qquad 1\mapsto 1
$$
from Example~\eqref{ex:point-sing} to~\eqref{ex:ES1-sing} are
quasi-isomorphisms because all homologies $H(C,\delta)$ equal $\R$. 
\end{proof}

Comparison of the homology groups in these three examples shows that
the $5$ remaining homologies in
Proposition~\ref{prop:quism-invariance} are not quasi-isomorphism
invariants of mixed complexes.

%%%%%%%%%%%%%%%%%%%%%%%%%%%%%%%%%%%%%%%%%%%%%%%%%%%%%%%%%%%%%%%%%%%%%%%%%%%%
\section{Chen's iterated integral}\label{sec:chen}
%%%%%%%%%%%%%%%%%%%%%%%%%%%%%%%%%%%%%%%%%%%%%%%%%%%%%%%%%%%%%%%%%%%%%%%%%%%%

In this section we study the behaviour of the various flavours of
cyclic homology under Chen's iterated integral, introduced by
K.T.~Chen in~\cite{Chen73,Chen77}. 

%%%
\subsection{Chen's iterated integral as a map of mixed complexes}
%%%
Let $X$ be a connected oriented manifold and $LX$ its free loop
space. For each $k\in\N$ the {\em Chen pairing} 
\begin{equation}\label{eq:chenpair}
   \la\cdot,\cdot\ra:\Om^{*+k}(X^{k+1})\times C_*(LX)\longrightarrow\R
\end{equation}
between differential forms on the $(k+1)$-fold product $X^{k+1}$ and
singular chains on $LX$ is defined as follows. Recall the standard
simplex
$$
   \Delta^k = \{(t_1,...,t_k)\in \R^k\mid 0\le t_1\le t_2\le...\le
   t_k\le 1\},\qquad \Delta^0=\{0\} 
$$
with its face maps (parametrizing the boundary faces)
$\delta_j:\Delta^{k-1}\longrightarrow \p\Delta^k$ defined by
\begin{gather*}
   \delta_0(t_1,...,t_{k-1})=(0,t_1,...,t_{k-1}),\qquad
   \delta_k(t_1,...,t_{k-1})=(t_1,...,t_{k-1},1), \cr
   \delta_j(t_1,...,t_{k-1})=(t_1,...,t_j,t_j,...,t_{k-1}),\qquad j=1,...,k-1. 
\end{gather*}
We give $\Delta^k$ the induced orientation from $\R^k$. Then
$\delta_0$ is orientation reversing, $\delta_1$ is orientation
preserving, and in general $\delta_j$ changes orientation by 
$(-1)^{j+1}$. Consider a simplex $B$ and a continuous map
$$
   f:B\longrightarrow LX.
$$
For any $k\ge 0$ we define the evaluation map
$$
   ev_f:B\times \Delta^k\longrightarrow X^{k+1},\quad
   ev_f(p,t_1,...,t_k) := \bigl(f(p)(0),f(p)(t_1),...,f(p)(t_k)\bigr). 
$$
Given $\om\in \Om^*(X^{k+1})$ and such a map $f$, we define the
pairing~\eqref{eq:chenpair} as 
$$
    \la\om,f\ra := \int_{B\times\Delta^k}ev_f^*\om.
$$
It gives rise to a degree preserving linear map, {\em Chen's iterated integral}
\begin{equation}\label{eq:chen}
   I:\bigoplus_{n\geq 0}\Om^{*+n}(X^{n+1})\to C^*(LX),\qquad (I\om)(f):=\la\om,f\ra. 
\end{equation}
Recall from Example~\ref{ex:forms} that the spaces $\Om^*(X^{n+1})$ for $n\in\N_0$ form a cyclic
cochain complex, so according to Lemma~\ref{lem:cyccomplex-dgmod} they
give rise to a mixed complex $(\bigoplus_{n\geq 0}\Om^*(X^{n+1}),d_H=d+b,B)$. 
The other side of equation~\eqref{eq:chen} carries the structure of a
mixed complex $(C^*(LX),d,P)$ provided by Lemma~\ref{lem:S1-space-dgmod}. 

\begin{prop}\label{prop:chen}
Chen's iterated integral $I$ defines a morphism of mixed complexes
$$
   I:\Bigl(\bigoplus_{n\geq 0}\Om^*(X^{n+1}),d_H,B\Bigr)\longrightarrow (C^*(LX),d,{-}P).
$$
\end{prop}

\begin{proof}
Consider $\om\in\Om^*(X^{k+1})$ and $f:B\to LX$ as above. 
For compatibility with the differentials, we need to show
\begin{equation}\label{eq:dH}
   \la d_H\om,f\ra = \la\om,\p f\ra.
\end{equation}
Using Stokes' theorem, we rewrite the left hand side as
\begin{align*}
   \la d_H\om,f\ra
   &= \int_{B\times \Delta^k}ev_f^*d\om + \int_{B\times
     \Delta^{k-1}}ev_f^*(b\om) \cr
   &= \int_{\p(B\times \Delta^k)}ev_f^*\om + \int_{B\times
     \Delta^{k-1}}ev_f^*(b\om) \cr
   &= \int_{\p B\times \Delta^k}ev_f^*\om + (-1)^{\dim B}\int_{B\times
     \p\Delta^k}ev_f^*\om + \int_{B\times \Delta^{k-1}}ev_f^*(b\om) \cr
   &= \la\om,\p f\ra + \left((-1)^{|\om|+1}\int_{B\times
     \p\Delta^k}ev_f^*\om + \int_{B\times \Delta^{k-1}}ev_f^*(b\om)\right).
\end{align*}
Here for the signs in the last step we have used that $\deg d\om =
\dim (B\times \Delta^k)$ and thus $\dim B = \deg\om + 1 - k = |\om| +
1$. So for~\eqref{eq:dH} it remains to show 
\begin{equation}\label{eq:b}
   \int_{B\times\Delta^{k-1}}ev_f^*(b\om)= ( -1)^{|\om|}\int_{B\times
     \p\Delta^k}ev_f^*\om.
\end{equation}
For this, observe that the face maps $\delta_j:X^k\to X^{k+1}$ from
Example~\eqref{ex:top} and the face maps
$\delta_j:\Delta^{k-1}\to \p\Delta^k$ satisfy for each 
$j=0,\dots,k$ the relation
$$
   ev_f\circ (\id\times\delta_j)=\delta_j\circ ev_f:
   B\times\Delta^{k-1}\to X^{k+1}.
$$
Using this and the fact that $\delta_j:\Delta^{k-1}\to \p\Delta^k$
changes orientation by $(-1)^{j+1}$ we obtain
$$
   \int_{B\times \Delta^{k-1}}ev_f^*\delta_j^*\om 
   = \int_{B\times \Delta^{k-1}}(\id\times\delta_j)^*ev_f^*\om 
   = (-1)^{j+1}\int_{B\times \delta_j(\Delta^{k-1})}ev_f^*\om. 
$$
Now recall that $b\om=\sum_{j=0}^k(-1)^{|\om|+j+1}\delta_j^*\om$ and
$\p\Delta^k=\bigcup_{j=0}^k\delta_j(\Delta^{k-1})$. Multiplying both
sides of the last displayed equation with $(-1)^{j+|\om|+1}$ and summing
over $j=0,\dots,k$ thus yields equation~\eqref{eq:b}. 

Next we show compatibility with the BV operators
\begin{equation}\label{eq:commBP}
   (I\circ B)(\om)(f) = ({-}P\circ I)(\om)(f)
\end{equation}
%\marginpar{$k$ changed to $k+1$ from here...May 2020 EV}
for $\om\in\Om^*(X^{k+1})$ and $f:B\to LX$.  
We first discuss the left hand side of \eqref{eq:commBP}. 
Recall that $B\om=(1-t)sN\om=(-1)^{|\om|}(1-t)\sigma_0^*N\om$, where
$\sigma_0:X^{k+2}\to X^{k+1}$ is the projection forgetting the zero-th factor 
and $t=(-1)^{k+1}\tau_{k+1}^*$ with $\tau_{k+1}:X^{k+2}\to X^{k+2}$ given by
$\tau_{k+1}(x_0,\dots,x_{k+1})=(x_1,\dots,x_{k+1},x_0)$. It follows that the form
$tsN\om\in\Om^*(X^{k+2})$ is independent of the variable $x_1$, so the
contraction of its pullback $ev_f^*tsN\om$ with the coordinate vector
field $\p_{t_1}$ on $\Delta^{k+1}$ is zero, and therefore the integrand
$ev_f^*tsN\om$ vanishes pointwise. This shows that $I(tsN\om)=0$, and thus
\begin{equation}\label{eq:lhs}
   (I\circ B)(\om)(f) = \la (1-t)sN\om,f\ra 
%   = (-1)^{|\om|}\int_{B\times\Delta^{k}}ev_f^*\sigma_0^*N\om
   = (-1)^{|\om|}\int_{B\times\Delta^{k+1}}\wh{ev_f}^*N\om
\end{equation}
with $\wh{ev_f}:=\sigma_0\circ ev_f:B\times \Delta^{k+1}\longrightarrow
X^{k+1}$. By definition of $ev_f$ and $\sigma_0$ we have
\begin{equation}\label{ev-hat}
   \widehat{ev_f}(p,t_1,...,t_{k+1}) = (f(p)(t_1),...,f(p)(t_{k+1})). 
\end{equation}
%\marginpar{...to here. Minor adjustments below. May 2020 EV}
Let us now rewrite the right hand side of \eqref{eq:commBP}. Recall
that the pushforward $\phi_*(F_{S^1}\times f)$ is given by the map
$\Delta^1\times B\rightarrow LX$ sending $(s,p)$ to the loop
$f(p)(\cdot+s)$. Switching the order of $\Delta^1$ and $B$, let us
define the map $S^1f:B\times S^1\rightarrow LX$ by
$S^1f(p,s)(t):=f(p)(t+s)$. Since switching the order changes
orientation by $(-1)^{\dim B}$, we obtain 
\begin{equation}\label{eq:rhs}
   (P\circ I)(\om)(f) 
   = \la\om,\phi_*(F_{S^1}\times f)\ra 
   = (-1)^{\dim B}\int_{B\times (\Delta^1\times\Delta^k)}ev_{S^1f}^*\om.
\end{equation}
To proceed further we need to relate the two maps $ev_{S^1f}$ and
$\widehat{ev_f}$ and their domains of definition. For this we split
$\Delta^1\times\Delta^k$ into the $k+1$ regions
$$
   R_i := \{(s,t_1,\dots,t_k)\mid t_{i-1}\leq 1-s \leq t_i\},\qquad i=1,\dots,k+1, 
$$
where we stipulate that $t_0=0$ and $t_{k+1}=1$.
Then $\bigcup_{i=0}^kR_i=\Delta^1\times\Delta^k$ with $R_i\cap R_j$ of
measure zero for $i\neq j$. In $R_i$, the condition $t_{i-1}+s\leq
1\leq t_i+s$ implies the ordering
$$
   0\leq t_i+s-1\leq\cdots\leq t_k+s-1\leq s\leq t_1+s\leq\cdots\leq
   t_{i-1}+s\leq 1,
$$ 
so $R_i$ is naturally diffeomorphic to $\Delta^{k+1}$ via the map
$r_i:R_i\longrightarrow \Delta^{k+1}$, 
$$
   r_i(s,t_1,...,t_k) := (t_i+s-1,\dots,t_k+s-1,s,t_1+s,\dots,t_{i-1}+s).
$$
Note that $r_i$ is orientation preserving if and only if $ik$ is even.
Unwrapping the definition of the various maps we find that they
satisfy the relation
$$
   ev_{S^1f}=\tau_k^{-i}\circ \widehat{ev_f}\circ (\id\times r_i),
$$
with the permutation map $\tau_k$ from above. Using $t_k=\tau_k^*$ and
and invariance of integration with respect to $\id\times r_i$, we deduce
\begin{align*}
   (-1)^{\dim B}\int_{B\times R_i}ev_{S^1f}^*\om
   &= (-1)^{\dim B}\int_{B\times R_i}(\id\times r_i)^*\widehat{ev_f}^*t_k^{-i}\om \cr
   &= (-1)^{\dim B+ik}\int_{B\times \Delta^{k+1}}\widehat{ev_f}^*(t_k^{-i}\om).
\end{align*}
Since $t=(-1)^kt_k$ satisfies $t^{k+1}=1$, we can write
$$
N=\sum_{j=0}^kt^j = \sum_{i=1}^kt^{k+1-i} = \sum_{i=1}^{k+1}t^{-i} = \sum_{i=1}^{k+1}(-1)^{ik}t_k^{-i}.
$$ 
Summing up the
previous displayed equation over $i=1,\dots,k+1$ and using
equation~\eqref{eq:rhs} we therefore obtain 
$$
   (P\circ I)(\om)(f) = (-1)^{\dim B}\int_{B\times \Delta^{k+1}}\widehat{ev_f}^*(N\om). 
$$
Now observe that $\dim B+k+1=\deg\om$, and thus $\dim B=|\om|-1$ since
$\om\in\Om^*(X^{k+1})$. Comparing with equation~\eqref{eq:lhs}, this finishes the
proof of equation~\eqref{eq:commBP} and thus of Proposition~\ref{prop:chen}.
%\marginpar{We are now off by a sign!. May 2020 EV}
\end{proof}

%%%
\subsection{Chen's iterated integral on Connes' version}
%%%

On Connes' version of cyclic homology, Chen's iterated integral is
induced by the {\em cyclic Chen pairing}
\begin{equation}\label{eq:cycchenpair}
   \la\cdot,\cdot\ra_\cyc:\Om^{*+k}(X^{k})\times C_*(LX)\longrightarrow\R.
\end{equation}
It is defined on $\om\in \Om^{*+k}(X^{k})$ and $f:B\to LX$ by
$$
   \la\om,f\ra_\cyc :=(-1)^{|\om|}\int_{B\times\Delta^{k}_{\rm cyc}}\wh{ev_f}^*\om,
$$
where $\Delta^{k}_{\rm cyc}$ denotes the space of tuples
$(t_1,\dots,t_k)\in (S^1)^k$ in the cyclic order $t_1\leq
t_2\leq\cdots\leq t_k\leq t_1$ and $\wh{ev_f}$ is defined
in~\eqref{ev-hat}. 
The {\em Connes (or cyclic) version of Chen's iterated integral} is
the degree preserving map
$$
   I^\lambda:\bigoplus_{n\geq 0}\Om^{*+n}(X^{n})\longrightarrow C^*(LX), \qquad
   I^\lambda(\om)(f) := \la\om,f\ra_\cyc.
$$

\begin{lemma}\label{lem:Ilambda}
%The iterated integral $I_\lambda$ has the following properties:
(a) $I^\lambda$ is given by the composition
$$
   I^\lambda:\bigoplus_{n\geq 0}\Om^{*+n}(X^{n})\stackrel{B}\longrightarrow
   \bigoplus_{n\geq 0}\Om^{*+n}(X^{n+1})\stackrel{I}\longrightarrow C^*(LX).
$$
In particular, $I^\lambda=I\circ B: \bigl(C_*(\Om(X)),d+b\bigr)\longrightarrow
\bigl(C^*(LX),d\bigr)$ is a chain map. 

(b) $I^\lambda$ vanishes on $\im(1-t)$ and thus descends to
$C_*^\lambda(\Om(X))=C_{*+1}(\Om(X))/\im (1-t)$. 

(c) The following diagram commutes, where $\pi_0$ denotes the quotient
   projection of the constant term in $u^{-1}$:
\begin{equation*}
\begin{CD}
   C_{*+1}(\Om(X))[u^{-1}] @>{\pi_0}>> C_*^\lambda(\Om(X)) @>{B}>> C_*(\Om(X))[[u]]\\
   @V{I}VV @VV{I^\lambda}V @VV{I}V \\
   C^{*+1}(LX)[u^{-1}] @>{{-}P\pi_0}>> C^*(LX) @>{\iota_0}>> C^*(LX)[[u]].
\end{CD}
\end{equation*}
(d) For smooth maps $f:B\to LX$ and $\sigma:B\to S^1$ define
$f_\sigma:B\to LX$ by $f_\sigma(p)(t):=f(p)(t+\sigma(p))$. Then for
each $\om\in\Om^*(X^k)$ we have
$$
   \la I^\lambda\om,f_\sigma\ra = \la I^\lambda\om,f\ra. 
$$
\end{lemma}

\begin{proof}
For part (a) we decompose
$$
   \Delta^n_\cyc = \bigcup_{i=1}^n\Delta^n_i,\qquad
   \Delta^n_i:=\{(t_1,\dots,t_n)\mid t_i\leq 0\leq t_{i+1}\}
$$
where each $\Delta^n_i$ is isomorphic to the standard simplex via the
permutation map $\tau_{n-1}^i:\Delta^n\stackrel{\cong}\to\Delta^n_i$.
Now for $\om\in\Om^{*+n}(X^n)$ and $f:B\to LX$ we compute
\begin{align*}
   \la\om,f\ra_\cyc
   &= (-1)^{|\om|}\int_{B\times\Delta^n_{\rm cyc}}\wh{ev_f}^*\om 
   = (-1)^{|\om|}\sum_{i=1}^n(-1)^{(n-1)i}\int_{B\times\Delta^n}\wh{ev_f}^*(\tau_{n-1}^i)^*\om \cr
   &= (-1)^{|\om|}\int_{B\times\Delta^n}\wh{ev_f}^*(\sum_{i=1}^nt^i\om) 
   = (-1)^{|\om|}\int_{B\times\Delta^n}\wh{ev_f}^*(N\om) 
   = \la B\om,f\ra,
\end{align*}
where the last equality follows from~\eqref{eq:lhs}. This proves  
$I^\lambda=I\circ B$, and this is a chain map because $I$ (by
Proposition~\ref{prop:chen}) and $B$ are chain maps. 

Part (b) follows from $B(1-t)=0$, which holds because $B=(1-t)sN$ and
$N(1-t)=0$. In part (c) commutativity of the first square follows by
applying $\pi_0$ to $I^\lambda=I\circ B={-}P\circ I$ (which holds by
Proposition~\ref{prop:chen}), and commutativity of the second square
follows from part (a). For part (d), note that $ev_{f_\sigma} =
ev_f\circ\rho$ with the orientation preserving diffeomorphism
\begin{align*}
   \rho: B\times \Delta^k_\cyc &\to B\times \Delta^k_\cyc, \qquad
   (p,t_1,\dots,t_k) &\mapsto
   \Bigl(p,t_1+\sigma(p),\dots,t_k+\sigma(p)\Bigr), 
\end{align*}
so part (d) follows from invariance of integration under diffeomorphisms.
\end{proof}

Consider the following diagram: 
\begin{equation*}
\xymatrix{
& LX\times ES^1 \ar[d]^{\pi} \ar[r]^-{pr_1} & LX \\
B \ar[ur]^{\wt g} \ar[r]^-{g} & LX\times_{S^1}ES^1 \\
}
\end{equation*}
Using this, we define 
$$
   \la \ol I^\lambda\om,g\ra := \la I^\lambda\om,pr_1\circ\wt g\ra
$$
for $\om\in\Om^*(X^k)$ and $\wt g:B\to LX\times ES^1$ any lift of
$g:B\to LX\times_{S^1}ES^1$, which exists for each simplex $B$ because
$\pi:LX\times ES^1\to LX\times_{S^1}ES^1$ is a circle bundle. Any
other lift of $g$ is of the form $\wt g_\sigma(p)=\sigma(p)\cdot\wt
g(p)$ for a smooth map $\sigma:B\to S^1$ (where $\cdot$ denotes the
circle action on $LX\times ES^1$), and applying 
Lemma~\ref{lem:Ilambda} to $pr_1\circ\wt g_\sigma=(pr_1\circ\wt
g)_\sigma$ we see that the definition does not depend on the lift $\wt
g$ and defines a chain map
$$
   \ol I^\lambda: \bigl(C_*^\lambda(\Om(X)),d+b\bigr)\longrightarrow
   \bigl(C^*(LX\times_{S^1}ES^1),d\bigr).
$$
Note that the homology of the right hand side is the 
equivariant cohomology $H^*_{S^1}(LX)=H^{*}(LX\times_{S^1}ES^1,d)$ of
$LX$ defined via the Borel construction. 
Passing to homology in Lemma~\ref{lem:Ilambda}(c) we thus obtain
the following commuting diagram, where the maps $\pi_{0*}$ and
$\iota_{0*}$ are isomorphisms by Lemma~\ref{lem:cyccomplex-computation} 
and Proposition~\ref{prop:Borel}, respectively:
\begin{equation}\label{eq:Ilambda}
\begin{CD}
   HC_{*+1}^{[u^{-1}]}(\Om(X)) @>{\pi_{0*}}>{\cong}> HC_*^\lambda(\Om(X)) @>{B_*}>> HC_*^{[[u]]}(\Om(X))\\
   @VV{I_*}V @VV{I^\lambda_*}V @VV{I_*}V \\
   H^{*+1}_{[u^{-1}]}(LX) @>{{-}P_{0*}}>> H^*_{S^1}(LX) @>{\iota_{0*}}>{\cong}> H^*_{[[u]]}(LX).
\end{CD}
\end{equation}

%%%
\subsection{Chen's iterated integral for simply connected manifolds}
%%%

In this subsection we assume in addition that $X$ is simply connected. 
This has two important consequences. First, by Goodwillie's
Theorem~\ref{thm:Goodwillie} the projection $LX\to\pt$ to a
point induces an isomorphism $H^*_{[[u,u^{-1}]}(LX) \cong \R[u,u^{-1}]$.
For the second consequence, recall that composition of the map $I$ from
Proposition~\ref{prop:chen} with the morphism $\phi$ from
Corollary~\ref{cor:alg-ana} yields a morphism of mixed complexes 
(still denoted $I$) 
$
   I:\bigl(\bigoplus_{n\geq 0}\Om^*(X)^{\otimes n+1},d_H,B\bigr)\to (C^*(LX),d,{-}P).
$
The following theorem is proved in~\cite{Getzler-Jones-Petrack} and
attributed to Chen, see also Jones~\cite{Jones}. 

\begin{theorem}[Getzer, Jones and Petrack~\cite{Getzler-Jones-Petrack}]\label{thm:GJP}
Let $X$ be a simply connected manifold. Then Chen's iterated integral
$$
   I:\Bigl(C_*(\Om(X)),d_H,B\Bigr)\longrightarrow (C^*(LX),d,{-}P)
$$
induces an isomorphism on Hochschild homology, hence defines a
quasi-isomorphism of mixed complexes. 
\end{theorem}

\begin{corollary}\label{cor:GJP}
In the situation of Theorem~\ref{thm:GJP}, Chen's iterated integral
$I$ induces isomorphisms on the classical $[[u]]$, $[[u,u^{-1}]$ and
$[u^{-1}]$ versions of cyclic homology. 
On the other five versions it does not induce an isomorphism in general.
\end{corollary}

\begin{proof}
The first assertion follows directly from Theorem~\ref{thm:GJP}
combined with Proposition~\ref{prop:quism-invariance}.
For the second assertion, comparison of Examples~\ref{ex:point-deRham}
and~\ref{ex:point-sing} shows that $I$ does not induce isomorphisms
on the $[u]$, $[u,u^{-1}]$ and $[u,u^{-1}]]$ versions for $X=\{\pt\}$.
The remaining two versions fit into the commuting diagram of
tautological sequences
$$
\xymatrix{
  \cdots HC_*^{[[u]]}(\Om(X)) \ar[d]^{I_*}_\cong \ar[r]
   & HC_*^{[[u,u^{-1}]]}(\Om(X)) \ar[d]^{I_*} \ar[r] 
   & HC_*^{[[u^{-1}]]}(\Om(X))\cdots \ar[d]^{I_*} \\
   \cdots H^*_{[[u]]}(LX) \ar[r] & H^*_{[[u,u^{-1}]]}(LX) \ar[r]
   & H^*_{[[u^{-1}]]}(LX)\cdots \\
}
$$
Now $HC_*^{[[u,u^{-1}]]}(\Om(X))\cong HC_*^{[[u,u^{-1}]}(\Om(X))\cong
H^*_{[[u,u^{-1}]}(LX) \cong \R[u,u^{-1}]$ where the first
isomorphism follows from Corollary~\ref{cor:cyccochain}, the second
one is given by $I_*$, and the last one follows from Goodwillie's
theorem~\ref{thm:Goodwillie}. On the other hand, by Theorem~\ref{thm:JP}
we have $H^*_{[[u,u^{-1}]]}(LX)\cong H^*(X)\otimes\R[u,u^{-1}]$. For
$X$ noncontractible these two $\R[u]$-modules differ, hence the middle
map $I_*$ in the diagram is not an isomorphism. It follows from the
five-lemma that the right map $I_*$ is not an isomorphism either. 
\end{proof}

The following theorem, which is the main result of this paper,
describes the behaviour of Connes' version under Chen's iterated integral.

\begin{theorem}\label{thm:main-homol}
Let $X$ be a simply connected manifold. Then the kernel and image of
the iterated integral $I^{\lambda}_*:HC_*^\lambda(\Om(X)) \to
H^{*}_{S^1}(LX)$ are given by
%(where $s$ denotes upward grading shift by $1$)
\begin{align*}
   \ker I^{\lambda}_* &= {\rm span\,}\{[1],[1^3],[1^5],\dots\}\cong \R[u^{-1}], \cr 
   \im I^{\lambda}_* &= \ker\Bigl(\iota^*:H^*_{S^1}(LX)\to\R[u]\Bigr) \cong
   H^*_{S^1}(LX)\Bigl/\R[u]. 
\end{align*}
\end{theorem}

\begin{proof}
Consider the commuting diagram
\begin{equation*}
\xymatrix{
  & HC^\lambda_*(\Om(X)) \ar[r]^{I^{\lambda}_*} & H_{S^1}^{*}(LX) \ar[d]^{\iota_{0*}}_\cong \\
  H^{*+1}_{[[u,u^{-1}]}(LX) \ar[d]^\cong \ar[r]^{p_*} & H^{*+1}_{[u^{-1}]}(LX)
    \ar[u]^{\pi_0I_*^{-1}}_\cong \ar[r]^{{-}P_{0*}} & H^{*}_{[[u]]}(LX)
    \ar[d]^{\iota^*} \ar[r]^{\cdot u} & H^{*+2}_{[[u,u^{-1}]}(LX) \ar[d]^\cong \\ 
  \R[u,u^{-1}] \ar[r]^{p_*} & \R[u^{-1}] \ar[u]^{\pi^*} \ar[r]^0
       & \R[u] \ar[r]^{\cdot u} & \R[u,u^{-1}]\,.
} 
\end{equation*}
Here the lower two rows follow from~\eqref{eq:proj-point} for $Y=LX$
where $\pi^*$ is injective, $\iota^*$ is surjective, and the outer
vertical maps are isomorphisms by Goodwillie's Theorem~\ref{thm:Goodwillie}.
The upper square follows from diagram~\eqref{eq:Ilambda}, where $I_*$
is an isomorphism by Corollary~\ref{cor:GJP} and we have abbreviated
$\iota_0P_0$ as $P_0$. We read off
\begin{align*}
   \ker P_{0*} &= \pi^*\R[u^{-1}] = {\rm span\,}\{u^{-k}1\mid
   k\in\N_0\}, \cr 
   \im P_{0*} &= \ker\Bigl(\iota^*:H^{*-1}_{[[u]]}(LX)\to\R[u]\Bigr),
\end{align*}
%In view of the canonical isomorphism $H^{*-1}_{[[u]]}(LX)\cong
%H^{*-1}_{S^1}(LX)$ from Proposition~\ref{prop:Borel}
%and diagram~\eqref{eq:Ilambda}, in which the maps $I_*$ are
%isomorphisms by Corollary~\ref{cor:GJP},
and therefore
\begin{align*}
   \ker I^{\lambda}_* &= \pi_{0*}I_*^{-1}\ker P_{0*} \cr
   &= \pi_{0*}{\rm span\,}\{u^{-k}1+u^{-k+1}1^3+\cdots+1^{2k+1}\mid
     k\in\N_0\} \cr
   &= {\rm span\,}\{[1],[1^3],[1^5],\dots\}, \cr 
   \im I^{\lambda}_* &= \im P_{0*} 
   = \ker\Bigl(\iota^*:H^{*-1}_{S^1}(LX)\to\R[u]\Bigr). 
\end{align*}
Here the description of $I_*^{-1}\ker P_{0*}$ and its image under $\pi_{0*}$
follow from Example~\ref{ex:point-deRham}. 
\end{proof}

\begin{cor}\label{cor:main}
Let $X$ be a simply connected manifold. Denote by
$\ol{HC}^\lambda_*(\Om(X))$ the reduced Connes version of cyclic homology of the
de Rham complex of $X$, and by  
$H^*_{S^1}(LX,x_0)$ the $S^1$-equivariant cohomology of $LX$ relative
to a fixed constant loop $x_0$. Then Chen's iterated integral induces
an isomorphism
%(where $s$ denotes upward grading shift by $1$ and $u$
%is a formal variable of degree $2$)
\begin{align*}
%   H_\lambda^*(\Om(X))\Bigl/s\R[u^{-1}] \cong H^*_{S^1}(LX)\Bigl/\R[u]
%   = H^*_{S^1}(LX,x_0).
   I^{\lambda}_*: \ol{HC}^\lambda_*(\Om(X)) \stackrel{\cong}\longrightarrow H^*_{S^1}(LX,x_0)\,.
\end{align*}
\end{cor}

\begin{proof}
Passing to reduced homologies, the commuting diagram in the proof of
Theorem~\ref{thm:main-homol} simplifies to
\begin{equation*}
\xymatrix{
  & \ol{HC}^\lambda_*(\Om(X)) \ar[r]^{I^{\lambda}_*} & H_{S^1}^{*}(LX.x_0) \ar[d]^{\iota_{0*}}_\cong \\
  0 \ar[r] & H^{*+1}_{[u^{-1}]}(LX,x_0) \ar[u]^{\pi_0I_*^{-1}}_\cong \ar[r]^{{-}P_{0*}}_\cong &
    H^{*}_{[[u]]}(LX,x_0) \ar[r] & 0
} 
\end{equation*}
and the corollary follows. 
%This follows directly from the description of the image and kernel of
%$I_{\lambda*}$ , 
%\begin{align*}
%  \im I_{\lambda*} &= \ker\Bigl(\iota^*:H^*_{S^1}(LX)\to\R[u]\Bigr) = H^*_{S^1}(LX,x_0),\cr
%  \ker I_{\lambda*} &= \pi_{0*}I_*^{-1}\pi^*\R[u^{-1}],
%\end{align*}
%where the last equation implies $HC_\lambda^*(\Om(X))/\ker
%I_{\lambda*} = \ol{HC}_\lambda^*(\Om(X))$. 
\end{proof}

%%%
\subsection{Computations for spheres}
%%%

For a simply connected manifold $X$, the $S^1$-equivariant cohomology
of $LX$ can be computed via minimal models as follows, see e.g.~\cite{Basu-thesis}. 
Let $(M=\Lambda[x_1,x_2,\dots],d)$ be the minimal model of $X$. To it
we associate a mixed complex $(LM,d,B)$ by setting
$LM:=\Lambda[x_1,\bar x_1,x_2,\bar x_2\cdots]$ with new generators
$\ol x_i$ of degrees $\deg\ol x_i=\deg x_i-1$, defining $B$ as the
derivation satisfying $Bx_i=\ol x_i$ and $B\ol x_i=0$, and extending
$d$ from $M$ to $LM$ by the requirement $dB+Bd=0$, i.e., by defining
$d\ol x_i:=-B(dx_i)$. Then $(LM[u],d_u=d+uB)$ is the minimal model for
the Borel space $LX\times_{S^1}ES^1$, so it computes $H^*_{S^1}(LX)$. 

Recall that a manifold $X$ is {\em formal over $\R$} (in the sense of
rational homotopy theory) if its de Rham dga $\Om^*(X)$ is connected
to its cohomology $H^*(X)$ by a zigzag of quasi-isomorphisms,
cf.~\cite{Pavel-thesis}. By Proposition~\ref{prop:quism-cyclicchains}, 
all versions of cyclic homology of $\Om^*(X)$ can then be computed
using the dga (with trivial differential) $H^*(X)$.

\begin{example}[odd dimensional spheres]
The sphere $S^n$ with $n\geq 3$ odd has minimal model $\Lambda[a]$
with $\deg a=n$ and $da=0$. So the minimal model of $LS^n\times_{S^1}ES^1$ is
$\Lambda[a,\bar a,u]$ with differential $d_u\bar a=0$ and $d_u a=u\bar a$.  
It follows that 
$$
   H_{S^1}^*(LS^n) = \Lambda[\bar a,u]\Bigl/\la u\bar a\ra = \bar a\R[\bar a]\oplus\R[u],
$$
where $u$ acts trivially on the first summand. Relative to a point
$x_0\in S^n$ this becomes
$$
   H_{S^1}^*(LS^n,x_0) = \bar a\R[\bar a],\qquad \deg\bar a=n-1.
$$
Let us now compute the reduced Connes version of cyclic homology of $\Om^*(S^n)$.
Since $S^n$ is formal, we can compute this from its cohomology $H^*(S^n)=\R
1\oplus\R v$, where $\deg v=n$, 
%and by Corollary~\ref{cor:reduced-cyc-hom} we can use 
or rather the reduced cohomology $\ol H^*(S^n)=\R v$. For $k\geq 1$ denote by 
$v^k$ the word with $k$ letters $v$. Since $t(v^k)=(-1)^{(k-1)+n^2(k-1)}v^k=v^k$,
each $v^k$ survives in the quotient by cyclic permutation, and by 
%Corollary~\ref{cor:reduced-cyc-hom} 
definition of $\ol{HC}_*^\lambda$ (with trivial Hochschild differential) we get
$$
   \ol{HC}_*^\lambda(\Om^*(S^n)) = HC_*^\lambda(\ol H^*(S^n)) =
   v\R[v],\qquad |v|=n-1.
$$
This is compatible with Theorem~\ref{thm:main-homol} if $I^{\lambda}_*$
sends $v$ to $\bar a$. 
\end{example}

\begin{example}[even dimensional spheres]
The sphere $S^n$ with $n\geq 2$ even has minimal model $\Lambda[a,b]$ 
with $\deg a=n$, $\deg b=2n-1$ and $da=0$, $db=a^2$. So the minimal model of $LS^n\times_{S^1}ES^1$ is
$\Lambda[a,\bar a,b,\bar b,u]$ with differential $d_u a=u\bar a$,
$d_ub=a^2+u\bar b$, $d\bar a=0$ and $d\bar b=-2a\bar a$.   
It follows (cf.~\cite{Basu-thesis}) that 
$$
   H_{S^1}^*(LS^n) = \bar a\Bigl(\Lambda[\bar a,u]\Bigl/\la u,2a,a^2+u\bar
   b\ra\Bigr) \oplus \R[u] = \bar a\R[\bar b]\oplus\R[u],
$$
where $u$ acts trivially on the first summand. Relative to a point
$x_0\in S^n$ this becomes
$$
   H_{S^1}^*(LS^n,x_0) = \bar a\R[\bar b],\qquad \deg\bar a=n-1,\ \deg\bar b=2n-2.
$$
Again, we can compute the reduced Connes version of cyclic homology of $\Om^*(S^n)$ 
from $\ol H^*(S^n)=\R v$, where $\deg v=n$. Since $t(v^k)=(-1)^{(k-1)+n^2(k-1)}v^k=(-1)^{k-1}v^k$,
the word $v^k$ survives in the quotient by cyclic permutation iff $k$
is odd, and by definition of $\ol{HC}_*^\lambda$ (with trivial Hochschild differential) we get
$$
   \ol{HC}_*^\lambda(\Om^*(S^n)) = HC_*^\lambda(\ol H^*(S^n)) = v\R[v^2],\qquad |v|=n-1.
$$
This is compatible with Theorem~\ref{thm:main-homol} if $I^{\lambda}_*$
sends $v^{2i+1}$ to $\bar a\bar b^i$. 
\end{example}

%%%%%%%%%%%%%%%%%%%%%%%%%%%%%%%%%%%%%%%%%%%%%%%%%%%%%%%%%%%%%%%%%%%%%%%%%%
\section{Duality}\label{sec:duality}
%%%%%%%%%%%%%%%%%%%%%%%%%%%%%%%%%%%%%%%%%%%%%%%%%%%%%%%%%%%%%%%%%%%%%%%%%%

%%%%%%%%%%%%%%%%%%%%%%%%%%%%%%%%%%%%%%%%%%%%%%%%%%%%%%%%%%%%%%%%%%%%%%%%%%
\subsection{Generalities}\label{subsec:generalities}
%%%%%%%%%%%%%%%%%%%%%%%%%%%%%%%%%%%%%%%%%%%%%%%%%%%%%%%%%%%%%%%%%%%%%%%%%%

As before, all complexes are over the ground field $\R$.
Let $(C=\bigoplus_{k\in\Z}C_k,\delta)$ be a chain complex. We
dualize it in such a way that the result is 
a chain complex as well, i.e.
$$
   \Bigl(C^\vee:=\bigoplus_{k\in\Z}\Hom(C_{-k},\R),\delta^*\Bigr). 
$$
Suppose now that $(C_j,\delta_j)$, $j=1,2$ are two chain complexes and 
$$
   \la\ ,\ \ra:C_1\otimes C_2\longrightarrow \R
$$
is a bilinear pairing of degree $0$ such that 
the differentials are mutual adjoints with respect to the pairing,
$$
   \la\delta_1x,y\ra = \la x,\delta_2y\ra.
$$
The pairing naturally gives rise to two maps
$$
   f_{12}:C_1\longrightarrow C_2^\vee,\qquad f_{21}:C_2\longrightarrow C_1^\vee
$$
by $f_{12}(x):=\left<x,\cdot\right>$ and $f_{21}(y):=\left<\cdot,y\right>$.
Since the differentials are adjoints, the maps $f_{12}$
and $f_{21}$ are chain maps. 
Denoting by $\iota_j:C_j\into (C_j^\vee)^\vee$ $j=1,2$ the
canonical embeddings into the second dual, we then have the following
equalities of chain maps: 
\begin{equation}\label{eq:refl}
f_{21}=f_{12}^*\circ\iota_2\quad \text{and}\quad  f_{12}=f_{21}^*\circ\iota_1\,.
\end{equation}
%Indeed, unwrapping definitions we get $f_{12}^*\circ\iota_2(y)=\iota_2(y)\circ f_{12}$ and $\iota_2(y)\circ f_{12}(x)=
%f_{12}(x)(y)=\left<x,y\right>=(f_{21}(y))(x)$ and similarly for the second equality.
We say that a graded vector space $E=\bigoplus_{k\in \Z}E_k$ is 
{\em graded finite dimensional} if each $E_k$ is finite dimensional. 

\begin{lemma}\label{lem:adjequiv}
In the above setting, assume that the homologies of $C_1$ and $C_2$
are graded finite dimensional. Then 
$f_{12}$ is a quasi-isomorphism if and only if $f_{21}$ is.
\end{lemma}

\begin{proof}
For each chain complex $C$ over $\R$ we have a commuting diagram
$$
   \xymatrix{
   H(C) \ar[d]^{\iota_H} \ar[r]^{H\iota} & H\bigl((C^\vee)^\vee\bigr) \ar[d]^{\cong}\\
   \bigl(H(C)^\vee\bigr)^\vee \ar[r]^{\cong} & H(C^\vee)^\vee,
}
$$
where $H\iota$ is the map on homology induced by the canonical
embedding $\iota:C\into (C^\vee)^\vee$, $\iota_H$ is the canonical
embedding for $H(C)$, and the two isomorphisms come from the universal
coefficient theorems. If $H(C)$ is graded finite dimensional, then the
map $\iota_H$ is an isomorphism (because finite dimensional spaces are
reflexive), hence so is $H\iota$. This shows that in the situation of
the lemma both canonical embeddings $\iota_1$ and $\iota_2$ are
quasi-isomorphisms. In view of equation~\eqref{eq:refl} this implies
the lemma, recalling that the dual of a quasi-isomorphism is again a
quasi-isomorphism. 
\end{proof}

%%%%%%%%%%%%%%%%%%%%%%%%%%%%%%%%%%%%%%%%%%%%%%%%%%%%%%%%%%%%%%%%%%%%%%%%%%
\subsection{Duality of mixed complexes}\label{subsec:dual-mixed}
%%%%%%%%%%%%%%%%%%%%%%%%%%%%%%%%%%%%%%%%%%%%%%%%%%%%%%%%%%%%%%%%%%%%%%%%%%

We now generalize the preceding discussion to mixed complexes. The dual
mixed complex to $(C,\delta,D)$ is defined as $(C^\vee,\delta^*,D^*)$. 
Suppose now that $(C_j,\delta_j,D_j)$, $j=1,2$ are two mixed complexes and  
$$
   \la\ ,\ \ra:C_1\otimes C_2\longrightarrow \R
$$
is a bilinear pairing of degree zero such that 
both differentials are mutual adjoints with respect to the pairing,
$$
   \la\delta_1x,y\ra = \la x,\delta_2y\ra,\quad \text{and} \quad \la
   D_1x,y\ra = \la x,D_2y\ra.
$$
This implies that the maps $f_{12}$ and $f_{21}$ are morphisms of
mixed complexes, and Lemma~\ref{lem:adjequiv} yields

\begin{cor}\label{cor:adjequivmixed}
In the above setting assume that the homologies $H(C_1,\delta_1)$ and
$H(C_2,\delta_2)$ are graded finite dimensional.  
Then $f_{12}$ is a quasi-isomorphism of mixed complexes if and only if $f_{21}$ is.
\hfill$\square$
\end{cor}

Let now $(C,\delta,D)$ be a mixed complex. We want to investigate the
relation between the total complex of its dual and the dual of its
total complex. For concreteness, let us consider the version $C[u^{-1}]$.
We define a degree zero pairing
\begin{equation}\label{eq:duality-pairing}
   \la\ ,\ \ra:C^\vee[[u]]^{-k}\otimes C[u^{-1}]_k\to\R,\qquad
   \la\phi,c\ra := \sum_{i\geq 0}\phi_i(c_{-i})
\end{equation}
where $\phi=\sum_{i\geq 0}\phi_iu^i$ with $\phi_i\in
(C^\vee)^{-k-2i}=\Hom(C_{k+2i},\R)$, and $c=\sum_{i\geq
  0}c_{-i}u^{-i}$ with $c_{-i}\in C_{k+2i}$. Note that the sum
$\sum_{i\geq 0}\phi_i(c_{-i})$ is finite because only finitely many
$c_{-i}$ are nonzero. Direct computation yields

\begin{lemma}\label{lem:dualcomm}
The pairing~\eqref{eq:duality-pairing} induces via
$\iota(\phi)(c)=\la\phi,c\ra$ a chain isomorphism 
$$
   \iota:\Bigl(C^\vee[[u]],\delta^*+uD^*\bigr)\stackrel{\cong}\longrightarrow \Bigl(C[u^{-1}]^\vee,(\delta+uD)^*\Bigr)
$$
respecting the $\R[u]$-module structures with $|u|=2$ on both sides.
Similarly, we obtain the chain isomorphisms
$$
   C^\vee[[u^{-1}]]\cong C[u]^\vee\quad\text{and}\quad C^\vee[[u,u^{-1}]]\cong C[u,u^{-1}]^\vee.
$$
\hfill$\square$
\end{lemma}

Finally, note that for a mixed complex $(C,\delta,D)$ and its dual we
have a commuting diagram of chain maps (with respect to $\delta^*$)
$$
\xymatrix{
   \im D^* \ar[d] \ar[r]^\cong & (\im D)^\vee \ar[r]_{D^*}^\cong & (C/\ker D)^\vee \ar[d]\\
   \ker D^* \ar[rr]_\phi^\cong & & (C/\im D)^\vee
}
$$
where the maps $D^*$ and $\phi$ have degree $-1$. On homology this yields
\begin{equation}\label{eq:im-ker-D}
\xymatrix{
   H(\im D^*) \ar[d] \ar[r]^\cong & H(\im D)^\vee \ar[r]_{(D^*)_*}^\cong & H(C/\ker D)^\vee \ar[d]\\
   H(\ker D^*) \ar[rr]_{\phi_*}^\cong & & H(C/\im D)^\vee.
}
\end{equation}

\subsection{Equivariant homology of $S^1$-spaces}\label{subsec:hom-S1}
%%%%%%%%%%%%%%%%%%%%%%%%%%%%%%%%%%%%%%%%%%%%%%%%%%%%%%%%%%%%%%%%%%%%%%%%%%

Let $Y$ be a topological $S^1$-space. It was shown in the proof of
Lemma~\ref{lem:S1-space-dgmod} that
$$
   (C_{-*}(Y),\p,Q)
$$
is a mixed complex, where $(C_*(Y),\p)$ is the singular chain complex and $Q$ 
is the map~\eqref{eq:Q} induced by the circle action. 
Note that we grade the singular chains negatively to give $\p$ and $Q$
degrees $1$ and $-1$, respectively. 
The homology of this mixed complex is the (negatively graded) singular homology
$H_{-*}(Y)$, and its dual is the mixed complex $(C^*(Y),d,P)$ in Lemma~\ref{lem:S1-space-dgmod}.
We denote the cyclic homology of $(C_{-*}(Y),\p,Q)$ by
$H_{-*}^{[[u]]}(Y)$ etc.

Lemma 5.1 in~\cite{Jones} and the fact that $C_{-*}(Y)$ lives in
nonpositive degrees imply the following dual version of
Proposition~\ref{prop:Borel}. 

\begin{proposition}[Jones~\cite{Jones}]\label{prop:Borel-dual}
For each topological $S^1$-space $Y$ we have canonical isomorphisms 
\begin{equation*}
   H^{S^1}_{-*}(Y) \cong H_{-*}^{[u^{-1}]}(Y) \cong H_{-*}^{[[u^{-1}]]}(Y).
\end{equation*}
\hfill$\square$
\end{proposition}

As in Section~\ref{subsec:circle-to-mixed}, for an $S^1$-space $Y$ with
a fixed point $y_0$ the inclusion and projection
$\pt\stackrel{\iota}\longrightarrow Y\stackrel{\pi}\longrightarrow\pt$,
$\iota(\pt)=y_0$, induce the following commuting diagram, where all the vertical maps
$\iota_*$ are injective and we have surjective vertical maps $\pi_*$
in the other direction:
\begin{equation}\label{eq:proj-point-dual}
\begin{CD}
   H_{-*}^{[[u,u^{-1}]}(Y) @>{p_*}>> H_{-*}^{[u^{-1}]}(Y) @>{Q_0}>>
       H_{-*-1}^{[[u]]}(Y) @>{\cdot u}>> H_{-*+1}^{[[u,u^{-1}]}(Y) \\
  @AA{\iota_*}A @AA{\iota_*}A @AA{\iota_*}A @AA{\iota_*}A \\
  \R[u,u^{-1}] @>{p_*}>> \R[u^{-1}] @>0>>
       \R[u] @>{\cdot u}>> \R[u,u^{-1}]\;. 
\end{CD}
\end{equation}

\begin{remark}
Lemma~\ref{lem:free-action} has the following dual version:
If $Y$ is a smooth $S^1$-space without fixed points, then
$H_{-*}^{[u,u^{-1}]]}(Y)=0$ and we have canonical isomorphisms 
\begin{equation*}
\begin{CD}
   H_{-*}^{[u^{-1}]}(Y) @>{\cong}>> H_{-*+1}^{[[u^{-1}]]}(Y) @>{\cong}>> H_{-*}(Y/S^1) 
   @>{\cong}>> H_{-*}^{[u]}(Y).
\end{CD}
\end{equation*}
For $Y$ with fixed points, applying this to $Y\times ES^1$ provides an
alternative proof of Proposition~\ref{prop:Borel-dual} in the smooth case. 
\end{remark}

%%%%%%%%%%%%%%%%%%%%%%%%%%%%%%%%%%%%%%%%%%%%%%%%%%%%%%%%%%%%%%%%%%%%%%%%%%
\subsection{Finiteness}\label{subsec:finiteness}
%%%%%%%%%%%%%%%%%%%%%%%%%%%%%%%%%%%%%%%%%%%%%%%%%%%%%%%%%%%%%%%%%%%%%%%%%%

In this subsection we prove two finiteness results on homology. 

\begin{lemma}\label{lem:findga}
Let $(A,d)$ be a dga whose cohomology $H^*(A)$ is graded finite
dimensional. Assume in addition that $\dim H^0(A)=1$ and $\dim
H^1(A)=0$. Then the Hochschild homology $HH(A)$ is graded 
finite dimensional. 
\end{lemma}

\begin{proof}
Consider the word length filtration on the Hochschild complex. This
filtration is bounded from below  and exhaustive, therefore the
corresponding spectral sequence converges to $HH(A)$. It is enough
to show that  graded finite dimensionality holds for the second
page. The first page computes to $E_1^{p,q}=H^p(A^{\otimes(q+1)},d)$
and the second page to 
$$
   E_2^{p,q}=H^q(E_1,b)^p=H^q(E_1/D(A),b)^p
$$ 
where $q$ denotes the word length degree, $p$ the degree in $A$, and
$D(A)$ is the acyclic subcomplex generated by words with $1$ in some
positive slot considered in Section~\ref{sec:dga}. We will show that
the desired finite dimensionality holds even before we take the
homology with respect to $b$. Fix some degree $k=p-q$ for the chain
complex $E_1/D$ and write out the degree $k$ part of the complex,
$$
   (E_1/D)^k=\bigoplus_{p-q=k}(E_1/D)^{p,q}
$$
Since $H^1(A)=0$ and we have factored out $D(A)$, we have
$(E_1/D)^{p,q}=0$ for $p< 2q$. So the sum runs over $p$ that satisfy
$p\ge 2q$, in other words $k=p-q\ge q$. This leaves us with only
finitely many options for $q$. Therefore, we have only finitely many
nonzero summands in $(E_1/D)^k$ and thus $\dim (E_1/D)^k<\infty$. 
\end{proof}

%Note that Lemma~\ref{lem:findga} can be combined with
%Lemma~\ref{lem:finmixed} to conclude graded finite dimensionality for
%the $3$ versions of cyclic homology. 
Note that Lemma~\ref{lem:findga} applies in particular to $A=\Om^*(X)$
for a simply connected manifold $X$.

\begin{lemma}\label{lem:finloopspace}
If $X$ is a simply connected manifold, then $H_*(LX)$ and $H_*^{S^1}(LX)$
%as well as $HC_*(LM)[[u]]$, $HC_*(LM)[[u,u^{-1}]$ and $HC_*(LM)[u^{-1}]$
are graded finite dimensional.
\end{lemma}

\begin{proof}
Consider the Sullivan minimal model $M=(\Lambda[x_1,x_2,\dots],d)$ for
$X$, where $\Lambda[x_1,x_2,\dots]$ is the free graded commutative
algebra on generators $x_i$ of degrees $\deg x_i\ge 2$. 
Moreover, by~\cite[Proposition 12.2]{Felix-Halperin-Thomas}
there are only finitely many generators $x_i$ of any given degree.
Then the minimal model for $LX$ is $LM=(\Lambda[x_1,\bar x_1,x_2,\bar
  x_2,\dots],\delta)$, with $\deg\bar x_i=\deg x_i-1$ and a suitable
differential $\delta$. Since all generators of $LM$ have
strictly positive degrees, $LM$ is graded finite dimensional, hence so
is its homology $H_*(LX)$. Graded finite dimensionality of $H_*^{S^1}(LX)$
follows by the same argument from its minimal model $(\Lambda[x_1,\bar x_1,x_2,\bar
  x_2,\dots,u],\delta_u)$, with $\deg u=2$ and a suitable
differential $\delta_u$. 
\end{proof}

%%%%%%%%%%%%%%%%%%%%%%%%%%%%%%%%%%%%%%%%%%%%%%%%%%%%%%%%%%%%%%%%%%%%%%%%%%
\subsection{Cyclic cohomology}\label{subsec:cyc-coh}
%%%%%%%%%%%%%%%%%%%%%%%%%%%%%%%%%%%%%%%%%%%%%%%%%%%%%%%%%%%%%%%%%%%%%%%%%%

Consider a mixed complex $(C,\delta,D)$ and its dual mixed complex
$(C^\vee,\delta^*,D^*)$. The {\em cyclic cohomology} of $C$ is defined as
$$
   HC_{[u,u^{-1}]]}^* := H(C^\vee)_{-*}^{[u,u^{-1}]]},
$$
and similarly for the other seven versions. Lemma~\ref{lem:dualcomm}
and the universal coefficient theorem yield
\begin{equation}\label{eq:dualcomm}
   HC^k_{[[u]]} = H(C^\vee)^{[[u]]}_{-k} \cong
     \bigl((HC^{[u^{-1}]})^\vee\bigr)_{-k} = \Hom(HC_k^{[u^{-1}]},\R),
\end{equation}
and similarly for the other two versions in Lemma~\ref{lem:dualcomm}.
Thus results about polynomial versions of cyclic homology
dualize to results about the corresponding power series versions of
cyclic cohomology. 

Consider now a cyclic cochain complex $(C_n,d_i,s_j,t_n,d)$ with its
associated mixed complex $(C:=\bigoplus_{n\geq 0}C_n,d+b,B)$. Its {\em
  Connes' version of cyclic cohomology} is defined as
\begin{align*}
   HC^*_\lambda &:= H_{-*-1}\bigl((C/\im(1-t))^\vee,d^*+b^*\bigr) 
   = H_{-*-1}\bigl(\ker(1-t)^*,d^*+b^*\bigr) \cr
   &\cong H_{-*-1}\bigl(C^\vee/\im(1-t^*),d^*+b^*\bigr),
\end{align*}
where the last isomorphism is induced by the inverse of the chain isomorphism
$N^*:C^\vee/\im(1-t^*)\stackrel{\cong}\longrightarrow \im N^*=\ker(1-t^*)\subset C^\vee$.
Recall from Corollary~\ref{cor:cyccochain}
the series of canonical isomorphisms
\begin{equation}\label{eq:cyccomplex-isos}
   HC_*^{[u^{-1}]}\cong HC_*^{[[u^{-1}]]}\cong HC_{*-1}^\lambda\cong HC_{*-1}^{[u]}.
\end{equation}
In view of equation~\eqref{eq:dualcomm}, dualizing the first, third
and fourth terms yields the isomorphisms
$$
   HC^{*-1}_{[[u^{-1}]]} \cong HC^{*-1}_\lambda \cong HC^*_{[[u]]}.
$$
Moreover, the proof of Lemma~\ref{lem:cyccomplex-computation} (which
uses only exactness of the rows in the $\theta$ double complex)
carries over to $C^\vee$ to yield the isomorphisms
$$
   HC^{*-1}_{[[u^{-1}]]} \cong HC^{*-1}_\lambda \cong HC^*_{[u]}.
$$
Combining these, we have proved

\begin{lemma}\label{lem:cyccomplex-dual}
For a cyclic cochain complex $(C_n,d_i,s_j,t_n,d)$, the canonical maps
on cyclic cohomology give the series of isomorphisms 
$$
   HC^{*-1}_{[[u^{-1}]]} \cong HC^{*-1}_\lambda \cong HC^*_{[u]} \cong HC^*_{[[u]]}.
$$
\hfill$\square$
\end{lemma}

Note that this series of isomorphisms differs from~\eqref{eq:cyccomplex-isos}
in the degrees and by the appearance of the $[u]$ rather than the
$[u^{-1}]$ version.

%%%%%%%%%%%%%%%%%%%%%%%%%%%%%%%%%%%%%%%%%%%%%%%%%%%%%%%%%%%%%%%%%%%%%%%%%%
\subsection{Chen's iterated integral on cyclic cohomology}\label{subsec:Chen-coh}
%%%%%%%%%%%%%%%%%%%%%%%%%%%%%%%%%%%%%%%%%%%%%%%%%%%%%%%%%%%%%%%%%%%%%%%%%%

Let $X$ be a manifold and $\Om(X)$ its de Rham dga. Recall the mixed
complexes $\bigl(C_*(\Om(X)),d+b,B\bigr)$ from Corollary~\ref{cor:alg-ana}, and 
$\Bigl(C_{-*}(LX),\p,Q\bigr)$ from Section~\ref{subsec:hom-S1} for $Y=LX$.
Let
$$
   \la\cdot,\cdot\ra:C_*(\Om(X))\otimes C_{-*}(LX)\to\R
$$
be the Chen pairing~\ref{eq:chenpair} from Section~\ref{sec:chen}. By
the proof of Proposition~\ref{prop:chen} this pairing respects the structures of
mixed complexes, so it induces two maps of mixed complexes: 
Chen's iterated integral $I:C_*(\Om(X))\to C^*(LX)$, and its adjoint
$$
   J:\bigl(C_{-*}(LX),\p,{-}Q\bigr)\to \bigl(C^{-*}(\Om(X)),d^*+b^*,B^*\bigr).
$$
Similarly, the cyclic Chen pairing $\la\ ,\ \ra_\cyc$ defined in~\eqref{eq:cycchenpair}
induces two chain maps: Connes' version of Chen's iterated integral $I^\lambda$,
and its adjoint
$$
   J_\lambda:\bigl(C_{-*+1}(LX),\p\bigr)\to \bigl(C^{-*}(\Om(X)),d^*+b^*\bigr).
$$
Lemma~\ref{lem:Ilambda} and the discussion following it dualize to

\begin{lemma}\label{lem:Jlambda}
(a) $J_\lambda$ is given by the composition of chain maps
$$
   J_\lambda:C_{-*+1}(LX)\stackrel{J}\longrightarrow C^{-*+1}(\Om(X))
   \stackrel{B^*}\longrightarrow C^{-*}(\Om(X)).
$$
(b) $J_\lambda$ lands in $\bigl(C_*(\Om(X))/\im(1-t)\bigr)^\vee = \ker(1-t^*)=C_\lambda^{-*+1}(\Om(X))$
and induces a map
%$\ol J_\lambda$ fitting into the commuting diagram
%$$
%\xymatrix{
%   C_{-*}(LX) \ar[r]^{J_\lambda} & \bigl(C^*(\Om(X))/\im(1-t)\bigr)^\vee \ar@{=}[d] \\
%   C_{-*}(LX\times ES^1/S^1) \ar[u] \ar[r]^{\ol J_\lambda} & \ker(1-t^*).
%}
%$$
$$
   \ol J_\lambda: C_{-*}(LX\times_{S^1}ES^1)\to \ker(1-t^*).
$$
(c) The following diagram commutes:
%, where $\pi_0$ denotes the quotient projection of the constant term in $u^{-1}$:
\begin{equation*}
\begin{CD}
   C^{-*+1}_{[u^{-1}]}(\Om(X)) @>{B^*\pi_0}>> C^{-*+1}_\lambda(\Om(X)) @>{\iota_0}>> C^{-*}_{[[u]]}(\Om(X))\\
   @A{J}AA @AA{J_\lambda}A @AA{J}A \\
   C_{-*+1}^{[u^{-1}]}(LX) @>{\pi_{0}}>> C_{-*+1}(LX) @>{{-}\iota_0Q}>> C_{-*}^{[[u]]}(LX).
\end{CD}
\end{equation*}
\hfill$\square$
\end{lemma}

The map $\ol J_\lambda$ induces a map on homology which we denote by 
$$
   J_{\lambda*}:H_*^{S^1}(LX)\to HC^*_\lambda(\Om(X)).
$$
Passing to homology in Lemma~\ref{lem:Jlambda}(c) we obtain
the following commuting diagram, where the maps $B_*$ and
$P_{0*}$ are isomorphisms by Lemma~\ref{lem:cyccomplex-dual}
and Proposition~\ref{prop:Borel-dual}, respectively:
\begin{equation}\label{eq:Jlambda}
\begin{CD}
   HC^{-*+1}_{[u^{-1}]}(\Om(X)) @>{B_*^*\pi_{0*}}>>
   HC^{-*+1}_\lambda(\Om(X)) @>{\iota_{0*}}>{\cong}> HC^{-*}_{[[u]]}(\Om(X))\\ 
   @AA{J_*}A @AA{J_{\lambda*}}A @AA{J_*}A \\
   H_{-*+1}^{[u^{-1}]}(LX) @>{\pi_{0*}}>{\cong}> H_{-*+1}^{S^1}(LX) @>{{-}\iota_{0*}Q_*}>> H_{-*}^{[[u]]}(LX).
\end{CD}
\end{equation}

{\bf The simply connected case. }
Assume now in addition that $X$ is simply connected. 
Then by Theorem~\ref{thm:GJP}, Chen's iterated integral
$I:C_*(\Om(X))\to C^*(LX)$ is a quasi-isomorphism of mixed complexes. 
Since the homologies $HH_*(\Om(X))$ (by Lemma~\ref{lem:findga}) and
$H_*(LX)$ (by Lemma~\ref{lem:finloopspace}) are graded finite dimensional,
Corollary~\ref{cor:adjequivmixed} applied to the Chen pairing yields

\begin{corollary}\label{cor:GJP-dual}
Let $X$ be a simply connected manifold. Then the dual Chen iterated integral
$J:C_{-*}(LX)\to C^{-*}(\Om(X))$ defines a quasi-isomorphism of mixed complexes,
and therefore induces isomorphisms on the $[[u]]$, $[[u,u^{-1}]$ and
$[u^{-1}]$-versions of cyclic cohomology. 
\hfill$\square$
\end{corollary}

As in the proof of Theorem~\ref{thm:main-homol}, we obtain a commuting diagram
\begin{equation*}
\xymatrix{
  & H^{S^1}_{*}(LM) \ar[r]^{J_{\lambda*}} & HC_\lambda^{*}(\Om(X)) \ar[d]^{\iota_{0*}J_*^{-1}}_\cong \\
  H_*^{[[u,u^{-1}]}(LX) \ar[d]^\cong \ar[r]^{p_*} & H_*^{[u^{-1}]}(LX)
    \ar[u]^{\pi_{0*}}_\cong \ar[r]^{{-}\iota_{0*}Q_{0*}} & H_{*-1}^{[[u]]}(LX)
    \ar[d]^{\pi_*} \ar[r]^{\cdot u} & H_{*+1}^{[[u,u^{-1}]}(LX) \ar[d]^\cong \\ 
  \R[u,u^{-1}] \ar[r]^{p_*} & \R[u^{-1}] \ar[u]^{\iota_*} \ar[r]^0
       & \R[u] \ar[r]^{\cdot u} & \R[u,u^{-1}]\,.
} 
\end{equation*}
Here the lower two rows follow from~\eqref{eq:proj-point-dual} for $Y=LX$
where $\iota_*$ is injective, $\pi_*$ is surjective, and the outer
vertical maps are isomorphisms by the dual version of Goodwillie's Theorem~\ref{thm:Goodwillie}.
The upper square follows from diagram~\eqref{eq:Jlambda}, where $J_*$
is an isomorphism by Corollary~\ref{cor:GJP-dual}.
%and we have abbreviated $\iota_0P_0$ as $P_0$.
We read off
\begin{align*}
   \ker J_{\lambda*} &= \pi_{0*}\iota_*\R[u^{-1}], \cr
   \im J_{\lambda*} &= \ker\Bigl(\pi_*\iota_{0*}J_*^{-1}:HC_{*}^\lambda(\Om(X))\to\R[u]\Bigr),
\end{align*}
and passing to reduced homologies as in the proof of Corollary~\ref{cor:main} we conclude

\begin{cor}\label{cor:main-dual}
Let $X$ be a simply connected manifold. Denote by
$\ol{HC}_\lambda^*(\Om(X))$ the reduced Connes version of cyclic cohomology of the
de Rham complex of $X$, and by  
$H_*^{S^1}(LX,x_0)$ the $S^1$-equivariant homology of $LX$ relative
to a fixed constant loop $x_0$. Then Chen's iterated integral induces
an isomorphism
\begin{align*}
   J_{\lambda*}: H_*^{S^1}(LX,x_0) \stackrel{\cong}\longrightarrow \ol{HC}_\lambda^*(\Om(X))\,.
\end{align*}
\hfill$\square$
\end{cor}

%%%%%%%%%%%%%%%%%%%%%%%%%%%%%%%%%%%%%%%%%%%%%%%%%%%%%%%%%%%%%%%%%%%
%%%%%%%%%%%%%%%%%%%%%%%%%%% References %%%%%%%%%%%%%%%%%%%%%%%%%%%%
%%%%%%%%%%%%%%%%%%%%%%%%%%%%%%%%%%%%%%%%%%%%%%%%%%%%%%%%%%%%%%%%%%%

\bibliographystyle{abbrv}
\bibliography{./000_cyc}

\end{document}